\documentclass[11pt]{article}
\usepackage{latexsym}
\usepackage{amsmath}
\usepackage{amssymb}
\usepackage{amsthm}
\usepackage{epsfig}
\usepackage{float}
\usepackage{xcolor}
\usepackage{graphicx}
\usepackage{authblk}
\usepackage{bm}
\usepackage{accents}
\usepackage{enumitem}
\usepackage{mathtools}
\usepackage{graphicx, natbib}
\usepackage{tikz, mathrsfs}
\usepackage[toc,page]{appendix}
\usepackage{graphicx, booktabs}
\usepackage{bbm}
\usepackage{lineno}
\usepackage{caption}
\usepackage{subcaption}

\usepackage[top=1in, bottom=1in, left=1in, right=1in]{geometry}
\usepackage[colorlinks=true, linkcolor=blue, citecolor=black, urlcolor=blue]{hyperref}
\usepackage[normalem]{ulem}

\renewcommand{\P}{\mathbb{P}}
\newcommand{\E}{\mathbb{E}}

\newcommand{\eps}{\varepsilon}

\DeclareMathOperator*{\argmin}{arg\,min}

\newtheorem{theorem}{Theorem}[section]
\newtheorem*{theorem*}{Theorem}
\newtheorem{lemma}[theorem]{Lemma}

\newtheorem{model}{Model}

\newtheorem{corollary}[theorem]{Corollary}

\theoremstyle{definition}

\newtheoremstyle{myremark} 
    {\topsep}                    
    {\topsep}                    
    {\rm}                        
    {}                           
    {\bf}                        
    {.}                          
    {.5em}                       
    {}  

\theoremstyle{remark}
\newtheorem{remark}{Remark}[section]

\DeclareSymbolFont{rsfs}{U}{rsfs}{m}{n}
\DeclareSymbolFontAlphabet{\mathscrsfs}{rsfs}

\def\bb{{\boldsymbol b}}
\def\be{{\boldsymbol e}}

\def\bu{{\boldsymbol u}}
\def\bv{{\boldsymbol v}}
\def\bw{{\boldsymbol w}}
\def\bx{{\boldsymbol x}}
\def\by{{\boldsymbol y}}
\def\bz{{\boldsymbol z}}

\newcommand{\bbeta}{\scalebox{0.9}[1.0]{$\bm{\beta}$}}

\def\beps{{\boldsymbol \eps}}

\def\bxi{{\boldsymbol \xi}}

\def\bv{{\boldsymbol v}}
\def\tbv{{\Tilde{\boldsymbol v}}}
\def\bu{{\boldsymbol u}}

\def\be{{\boldsymbol e}}
\def\bw{{\boldsymbol w}}
\def\bb{{\boldsymbol b}}

\def\cS{{\mathcal S}}

\def\cS{{\mathcal S}}

\def\bb{{\boldsymbol b}}

\def\spa{\hspace{-.08em}}

\def\bzeta{{\boldsymbol \zeta}}

\def\ind{\mathbbm{1}}

\newcommand{\comment}[1]{}

\def\cSc{{{\mathcal S}^c}}

\usepackage{graphicx, amssymb,amsmath,amsthm,xcolor,enumitem,mathtools}
\numberwithin{equation}{section}

\title{Robustness of OLS to sample removals: \\Theoretical analysis and implications}
\author[*]{Eyar Azar}
\author[$\dagger$]{Michael J. Feldman}
\author[*]{Boaz Nadler}
\date{}
\affil[*]{Weizmann Institute of Science}
\affil[$\dagger$]{Stern School of Business, New York University}
\begin{document}
\maketitle

%

%

\begin{abstract}
For learned models to be trustworthy, it is essential to verify their robustness to perturbations in the training data.
Classical approaches involve uncertainty quantification via confidence intervals and bootstrap methods. In contrast, recent work proposed a more stringent form of robustness: stability to the removal of 
{\em any} subset of $k$ samples from the training set. 
In this paper, we present a theoretical study of this criterion for ordinary least squares (OLS). 
Our contributions are as follows: (1) Given $n$ i.i.d.\ training samples from a general misspecified model, we prove that with high probability, OLS is robust to the removal of any $k \ll n $ samples.
(2) For data of dimension $p$, OLS can withstand up to ${k\ll \sqrt{np}/\log n}$
sample removals while remaining robust and achieving the same error rate as OLS applied to the full dataset. 
Conversely, if $k$ is proportional to $n$, OLS is 
provably non-robust. 
(3) We revisit prior analyses that found several econometric datasets to be highly non-robust to sample removals. 
While this appears to contradict our results in (1), we demonstrate that the sensitivity is due to either heavy-tailed responses or correlated samples. Empirically, this sensitivity
is considerably attenuated 
by classical robust methods, such as linear regression with a Huber loss.
\end{abstract}

\section{Introduction}\label{sec:intro}

Assessing the robustness of learned models and their predictions is essential for the widespread adoption of artificial intelligence systems. The standard supervised learning workflow involves training a model on a labeled dataset, evaluating its performance, and deploying it to make predictions on new unlabeled samples. A key question is whether we can trust the learned model---for example, to remain stable under minor data perturbations. To address this, a variety of tools have been developed to assess the reliability of learned models and quantify their uncertainty. These include cross-validation, bootstrap methods, confidence intervals, outlier detection, and conformal prediction \citep{efron1994introduction,huber2009robust,conformal}.


The stability of a model is closely related to how subsets of the training data affect its learned parameters. 
Classical works quantified the effect of individual samples using the influence function \citep{hampel1974influence, cook80}, while the bootstrap assesses stability under the removal of random subsets of samples. 
In contrast, contemporary studies have emphasized the importance of evaluating how specifically selected subsets of training samples collectively affect a learned model \citep{koh2019accuracy, basu2020second,guu2023simfluence,huang2024approximations}. Clearly, the reliability of a learned model is called into question if removing a handful of training samples drastically changes the fitted parameters.  For instance, \citet{broderick2020automatic} showed that in several econometrics datasets, removing even a small fraction of the data (in some cases, less than 1\%),  significantly alters the learned model. 


In response, a stringent notion of robustness was recently proposed: stability to the removal of {\em any} subset of $k$ samples from the training  data \citep{broderick2020automatic,diakonikolas2023algorithmic}. Identifying the subset of $k$ training samples whose removal most alters a model is known as the most influential subset selection (MISS) problem \citep{basu2020second, hu2024}. The naive brute-force approach to finding this subset is to  retrain the model on all subsets of the data of size $n-k$. Clearly, this is computationally infeasible for even moderately sized datasets once $k$ exceeds $3$ or $4$ samples. Indeed, various problems related to the robustness of OLS to sample removals are computationally intractable
\citep{moitraprovably, price2022hardness}. 
Several recent works---including \citet{broderick2020automatic},
\citet{kuschnig2021hidden},
 \citet{HopkinsFreund}, 
 \citet{moitraprovably}, 
 {{\citet{hu2024}}}, 
 \citet{hopkins} and \citet{huang2024approximations}---developed and analyzed approximate methods to solve the MISS problem, and, more generally, to assess and certify the robustness to sample removals. This emerging area of research is termed {\it robustness auditing}.

This form of robustness raises two fundamental questions: 
under the standard statistical learning paradigm, where the training samples $(\bx_1, y_1), \ldots, (\bx_n, y_n)$ are independent and identically distributed  (i.i.d.) from a probability distribution $P(\bx,y)$, when is a learned model provably robust to sample removals? 
Conversely, when is a model non-robust purely due to the inherent randomness of the training data, even under the i.i.d.\ sample assumption and in the absence of outliers?
In this paper, we study these questions within the context of  ordinary least squares (OLS), one of the most widely used tools in data science.

\paragraph{Contributions.}
We analyze the robustness of OLS in a general setting where the distribution $P(\bx,y)$ of the data is arbitrary, subject to mild regularity conditions. 
In this case, the optimal predictor of $y$ given $\bx$ under squared loss, $\E[y | \bx]$, is generally nonlinear in $\bx$, implying that the linear model is misspecified. Let $\widehat \bbeta$ denote the OLS coefficients computed on the full training data, and $\widehat \bbeta_\cS$ those computed on a subset $\cS$ of size $|\cS| \geq n-k$. Under mild conditions on $P(\bx,y)$, we prove that $\widehat \bbeta_\cS$ concentrates around $\widehat \bbeta$ {\it uniformly} over all subsets $\cS$, provided that $ k \ll n$. 
We present asymptotic results of the form 
\begin{align}  \label{eq1234}
\max_{\cS \subseteq [n]: |\cS| \geq n - k} \|\widehat{\bbeta} - \widehat \bbeta_\cS\| \xrightarrow{\, p \,} 0 \quad \text{as} \quad n \to \infty,
\end{align}
as well as non-asymptotic concentration inequalities. 
In particular, when $k \ll n$, OLS is provably robust irrespective of the data dimension $p$; that is, robustness holds even when $p \asymp n$, a regime in which OLS is inconsistent. 
Furthermore,  for ${k\ll \sqrt{np}/\log n}$, OLS achieves the same error rate as on the full dataset. That is, the error in (\ref{eq1234}) 
is comparable to $\|\bbeta - \widehat \bbeta\|$, where $\bbeta$ is the best linear predictor.

Next, under the standard Gaussian linear model, $\by = X\bbeta + \beps$ with $X$ multivariate Gaussian, we derive sharper non-asymptotic bounds with explicit constants.
As outlined in 
Table \ref{tab:robustness}, 
depending on $n$, $p$ and $k$, OLS exhibits distinct robustness and consistency properties. In particular, it is provably non-robust when $k$ is proportional to $n$.
The analysis under the general model uses standard concentration results and union bounds. 
In contrast, the sharper bounds for the Gaussian linear model are obtained using Gaussian comparison inequalities. {These inequalities have been used previously in high-dimensional statistics (see, for example, \cite{raskutti2010restricted, obozinski2011support}), although their application to robustness under sample removals is (to our knowledge) novel. The key observation is that subsets of size $n-k$ are highly overlapping when $k\ll n$, making union bounds overly conservative.} This approach may be useful for studying the robustness of regression methods beyond OLS to sample removals. {See Section~\ref{sec:mod2} for further discussion and comparison with prior results.}

Our theoretical results have several implications for robustness auditing. 
If the removal of a handful of samples significantly affects the fitted coefficients, the 
source of this
non-robustness should be carefully 
investigated. In particular, it  
may stem from one or more of the following factors:
(1) the covariance matrix of the data is ill-conditioned, (2) training samples are not all i.i.d., or (3) the data is heavy-tailed. In other words, even under model misspecification, the probability of $\widehat \bbeta$ exhibiting extreme sensitivity to sample removals, as observed in several recent studies, is  exponentially small if the data is ``well-behaved''.  

To complement our theoretical analysis, we empirically investigate the robustness of OLS to sample removals
{on the benchmark Diabetes dataset \citep{efron2004least}} and on two widely cited econometric datasets: the Cash Transfers data \citep{Angelucci09} and the Nightlights data \citep{martinez2022much}.
{In the Diabetes dataset, both the features and the response exhibit light tails. As detailed in Appendix \ref{appendix:diabetes}, the empirical effect of adversarial sample removals on this well-behaved dataset aligns closely with the scaling predicted by our theory.
}
In contrast,
prior robustness audits of the two econometric datasets, including \citet{broderick2020automatic} and \citet{hopkins}, identified small subsets of samples whose removal reversed the sign of the estimated treatment effects.

At first glance, these empirical findings appear to contradict our theoretical results. However, we show that the observed non-robustness is driven by either heavy-tailed responses or by correlated samples. In the Cash Transfers data, it is attributable to the heavy-tailed distribution of the response variable $y$. 
Specifically, many of the samples identified for removal by the AMIP algorithm (Approximate Maximum Influence Perturbation; \cite{broderick2020automatic}) have response values far exceeding  the mean of $y$ over the full dataset---often by  5 or 10 standard deviations, and in some cases, by more than 20.  
This key pattern, which elucidates the lack of robustness, is not 
discussed in \citet{broderick2020automatic} or \citet{hopkins}.
In the Nightlights data, the removed samples are concentrated in only a few countries and years, and are therefore highly correlated. 

Next, we show empirically that this non-robustness is substantially mitigated by applying classical robust regression techniques.
In particular, fitting a linear model under the Huber loss, the estimated coefficients are considerably more stable to sample removals. 
For the Cash Transfers data, roughly two to six times as many samples need to be removed to reverse the sign of the estimated treatment effect. Moreover, the robust estimates 
closely align with the original OLS estimates 
in \cite{Angelucci09}, which seems to reaffirm their findings.
Similarly, for the Nightlights dataset, compared to OLS, Huber regression
requires over 50\% more sample removals
to reverse the study's conclusions.
%
In summary, the interpretation of a robustness audit---that a study's conclusion may be overturned by removing a handful of samples---requires careful assessment.

{Before concluding the introduction, we further position our work within the broader literature on robustness and high-dimensional statistics. As discussed above, several recent works propose methods to quantify how groups of training samples affect a learned model and its predictions. Perhaps the most popular approach involves classical influence functions \citep{hampel1974influence}. While influence functions, which are based on first-order infinitesimal perturbations, traditionally approximate the effect of a single sample, influence-function-based subset methods approximate the effect of removing multiple samples by adding  their individual effects. This approximation can be inaccurate for the worst-case perturbation $\Delta_k(\bv)$, since the joint effect of several removals can differ substantially from the sum of individual effects \citep{guu2023simfluence}. Recent works report such failures empirically, even for low-dimensional OLS \citep{huang2024approximations, hu2024}.

Turning to high-dimensional regression, notable contributions include asymptotic theory for high-dimensional $M$-estimators \citep{el2013robust, lei2018asymptotics, adomaityte2024high}, prediction risk for OLS and ridge regression \citep{dobriban2018high, hastie2022surprises}, and bootstrap or resampling methods for high-dimensional linear models \citep{el2018can, ma2014statistical, ma2022asymptotic, meng2021lowcon}. 
These works address consistency, prediction risk, limiting distributions, or random resampling.
In contrast, our focus is the worst-case change in the OLS estimator under adversarial removal of any subset of $k$ samples.
}

Finally, while this work focuses on OLS, it suggests several potential research directions. A natural question is whether the robustness guarantees established here can be extended to other regression methods, such as ridge or kernel regression. Of particular interest are methods that use robust loss functions, which are designed to limit the influence of outliers or heavy-tailed data. Sparse linear regression raises another question: how stable are the selected variables under sample removals? We leave these questions for future research.

\begin{table*}[t]
  \caption{
    Asymptotic robustness and consistency properties of OLS under four regimes defined by relationships among $k$, $p$, and $n$. {Robustness is in the sense of (\ref{eq1234}); consistency means $\|\widehat \bbeta - \bbeta\| \xrightarrow{\, p \,} 0$ as $n \rightarrow \infty$, where $\bbeta$ is the coefficient vector of the best linear predictor, defined in (\ref{eq:beta_def}).}
  }
  \label{tab:robustness}
  \centering
  \begin{tabular}{ccccc}
    \toprule
    Region 
    & 
    & 
    & 
    & 
    \midrule
    Robust      & \checkmark & \checkmark & $\times$ & $\times$ \\
    Consistent  & \checkmark & $\times$   & \checkmark & $\times$ \\
    \bottomrule
  \end{tabular}
\end{table*}

\section{Problem Setting}
\label{sec:pb}

Consider a regression problem with training data $(\bx_1, y_1), \ldots, (\bx_n, y_n) \in \mathbb{R}^{p+1}$ and $n \geq p$. Let $X \in \mathbb{R}^{n \times p}$ contain $\bx_1, \ldots, \bx_n$ as its rows and $\by \coloneqq (y_1, \ldots, y_n)^\top$ be the $n$ response values. Assuming $X$ has full column rank, the OLS estimator is  
\begin{align} \label{eq:OLS}
\widehat \bbeta \coloneqq \big(X^\top X\big)^{-1} X^\top \by. 
\end{align}
For $\cS \subseteq [n]$ of size $|\cS| \geq p$, let $X_\cS \in \mathbb{R}^{|\cS| \times p}$ be the submatrix of $X$  consisting of the rows indexed by $\cS$, and define $\by_\cS$ analogously. 
Assuming $X_\cS$ has full column rank, the OLS estimator computed on the subset 
$(X_\cS,\by_\cS)$
of the training data is
\begin{align} \label{eq:betas_def}
\widehat \bbeta_\cS \coloneqq \big(X_\cS^\top X_\cS\big)^{-1} X_\cS^\top \by_\cS  .
\end{align}
As in \cite{hopkins}, the robustness of $\widehat \bbeta$ to $k$ sample removals in a fixed direction $\bv \in \mathbb{S}^{p-1}$ is 
measured by the following quantity: 
\begin{align}
\label{eq:delta_v}
    \Delta_k(\bv) \coloneqq \max_{\cS \subset [n], |\cS| = n-k}  \langle  \widehat \bbeta - \widehat \bbeta_\cS, \bv \rangle  ,
\end{align}
 where $\langle \bu,\bv\rangle$ denotes the inner product of the vectors $\bu$ and $\bv$. 
 For example, $\Delta_k(\be_1)$ quantifies the sensitivity of the first regression coefficient to $k$ removals. If 
 $ \Delta_k(\text{sign}(\widehat \beta_1)\be_1) > |\widehat \beta_1|$, 
 then there exists a subset of $k$ samples whose removal changes the sign of $\widehat\beta_1$. 
This may have important implications: 
 if $x_1$ (the first coordinate of $\bx$) is a treatment indicator in  a study, then $\beta_1$
  represents the treatment effect. If a small number of sample removals reverses the sign of $\widehat \beta_1$, the estimated treatment effect is unreliable and trust in the model is questionable.

As discussed in Section \ref{sec:intro}, exact evaluation of $\Delta_k(\bv)$ is computationally infeasible, and 
various methods have been developed  to approximate it. 
In contrast, our work focuses on a complementary aspect of robustness: we investigate conditions under which  $\Delta_k(\bv)$ is either provably close to zero, or bounded away from zero, even when the data are i.i.d.\ from a given distribution. Specifically, we study the robustness of OLS to sample removals under two generative models: (1) a general setting where the joint distribution of  $\bx$ and  $y$ is arbitrary, subject to mild conditions, in which case the linear model is potentially misspecified, and (2) a standard Gaussian linear model.

We now formally describe these two models. 
For completeness, a precise definition of sub-Gaussian  random variables is provided below.

\begin{model} \label{mod1}
    Let $P(\bx,y)$ be a probability distribution over $\mathbb{R}^{p+1}$ such that 
    \begin{enumerate}[label=(\roman*)]
    \item $\bx$ is sub-Gaussian, zero mean, and has positive-definite covariance $\Sigma \coloneqq \E[\bx \bx^\top]$.
    \item $y  $ is sub-Gaussian.
    \end{enumerate}
\end{model}

As is well known, the optimal predictor of $y$ given $\bx$ under squared loss is the conditional mean 
$\E[y|\bx]$. Under Model \ref{mod1}, 
this function is in general non-linear in $\bx$ and challenging to accurately estimate, particularly in high dimensions
where $p\gg 1$. 
A standard approach is to fit a linear model, which, although potentially misspecified, often has low variance and may perform well in practice.
Under squared loss, the quantity of interest is the vector of coefficients of the best linear predictor: 
\begin{align} \label{eq:beta_def}
\bbeta \coloneqq \argmin_{\bb \in \mathbb{R}^p} \E
\left[
(y - \bb^\top \spa  \bx)^2
\right]
= \Sigma^{-1} \E[y\bx] . 
\end{align}

\begin{model} \label{mod2}
Let $y = \bbeta^\top \spa \bx + \varepsilon$, where $\bbeta \in \mathbb{R}^p$ is deterministic and
 \begin{enumerate}[label=(\roman*)]
    \item $\bx \sim \mathcal{N}(0,\Sigma)$ where $\Sigma$ is positive definite.    
    \item $\varepsilon$ is sub-Gaussian, mean zero, and independent of $\bx$.
    \end{enumerate}
\end{model}
Properties of OLS have been extensively studied in both classical and high-dimensional regimes. M-estimators---including OLS and regression with Huber's loss---are consistent 
under Model \ref{mod2}
when $p/n \rightarrow 0$ \citep{portnoy1984asymptotic}. In contrast, in high-dimensional settings where $p$ is comparable to $n$, such estimators are inconsistent \citep{huber1973robust, el2013robust}. Recent works, including \citet{dobriban2018high} and \citet{hastie2022surprises}, derive exact asymptotic prediction errors for OLS and ridge regression. 
To our knowledge, however, the robustness of OLS to sample removals has not been systematically studied.

\begin{remark}
Model  \ref{mod2}  is a particular instance of Model \ref{mod1}. Hence, whenever OLS is robust under Model \ref{mod1}, it is also robust under Model \ref{mod2}.
\end{remark}

    \begin{remark}
        The assumption that $\bx$ is zero-mean 
        is standard in theoretical statistics and machine learning, as it simplifies the analysis. 
        Our results extend to $\bx$ of the form $ (1, \widetilde \bx)$, where $\widetilde \bx$ is sub-Gaussian, zero mean, and has positive-definite covariance. In this case, the OLS estimator defined in (\ref{eq:OLS}) includes a constant intercept.   
        The generalization of our analysis to the case where $\bx$ has an arbitrary mean is left for future work.
%
\end{remark}

\paragraph{Definitions and Notations.} 
We denote the spectral norm of a 
positive-definite matrix $\Sigma$ by $\|\Sigma\|$ and its condition number by  $\kappa(\Sigma) \coloneqq \lambda_1(\Sigma)/\lambda_p(\Sigma)$. The sub-Gaussian and sub-exponential  norms of a real-valued random variable $z$, denoted by $\|z\|_{\psi_2}$ and $\|z\|_{\psi_1}$, are defined as
\begin{align}\|z\|_{\psi_\alpha} \coloneqq \inf\big\{ t > 0 : \mathbb{E}[\exp(|z|^\alpha / t^\alpha)] \leq e \big\}, \label{eq:subgdef} \end{align}
for $\alpha = 1, 2$. 
For a random vector $\bz \in \mathbb{R}^p$, these norms are defined as
\[\|\bz\|_{\psi_\alpha} \coloneqq \sup_{\bv \in \mathbb{S}^{p-1}} \|\bv^\top \bz \|_{\psi_\alpha} , \quad \alpha = 1, 2. \]
We say that $\bz$ is sub-Gaussian if $\|\bz\|_{\psi_2} < \infty$ and sub-exponential if $\|\bz\|_{\psi_1} < \infty$. The choice of the constant $e$ in \eqref{eq:subgdef} is arbitrary; for instance, \citet{Vershynin12} uses the value 2. We adopt $e$ so that for deterministic scalars $a \in \mathbb{R}$, $\|a\|_{\psi_1} = |a|$. See Sections 2 and 3 of \cite{Vershynin12} for equivalent definitions in terms of the decay of tail probabilities.


\section{Theoretical Results}
\label{sec:results}

Sections \ref{subsec:k,p<<n} and \ref{sec:mod2} study the robustness of OLS under Models \ref{mod1} and \ref{mod2}, respectively. Throughout this section, $C, c > 0$ denote absolute constants.
All proofs are deferred to the appendix.

\subsection{Robustness Guarantees for Model \ref{mod1}} 
\label{subsec:k,p<<n}
This section establishes the asymptotic robustness of OLS to sample removals as \(n \to \infty \) and \( k / n \to 0 \). To this end, we derive a non-asymptotic concentration inequality that holds for finite $n, p,$ and $k$.

Formally, for our asymptotic results, we consider a sequence of regression problems indexed by the sample size \( n \). Each instance consists of \( n \) i.i.d.\ samples \( (\bx_1, y_1), \dots, (\bx_n, y_n)  \) drawn from a distribution satisfying Model~\ref{mod1}, where the dimension  \( p = p_n \) and the number of removals  \( k = k_n \) may increase with \( n \). 
This framework covers both the classical setting where $p$ is fixed, and the high-dimensional setting where $n, p \rightarrow \infty$ together. In the latter case, we assume that $\|\Sigma^{-1/2}\|$, $\|\bx\|_{\psi_2}$, and $\|y\|_{\psi_2}$ are uniformly bounded in $n$.
 
\begin{theorem}  \label{thrm1}
Let $(\bx_1, y_1), \ldots, (\bx_n, y_n) $ be  i.i.d.\ samples from Model \ref{mod1}. 
Assume that as $n \rightarrow \infty$, $\|\Sigma^{-1/2}\|$, $\| \Sigma^{-1/2} \bx\|_{\psi_2}$, and $\|y\|_{\psi_2}$ are bounded. If  $\limsup \|\Sigma^{-1/2}\bx\|_{\psi_2}^2 \sqrt{p/n} < c$ and $k/n \rightarrow 0$, then 
\begin{align} \label{eq:thrm1}
\max_{\cS \subseteq [n], |\cS| \geq n-k}  
\|\widehat \bbeta -\widehat \bbeta_\cS\|\xrightarrow{\,p\,} 0 . \end{align} 
\end{theorem}

\begin{remark}
{If $p=o(n)$, the condition 
$\limsup \|\Sigma^{-1/2}\bx\|_{\psi_2}^2 \sqrt{p/n} < c$ immediately follows from the boundedness of $\|\Sigma^{-1/2}\|$ and $\|\bx\|_{\psi_2}$. } 
\end{remark}

Theorem \ref{thrm1} guarantees the
robustness of OLS to $k = o(n)$ sample removals, irrespective of the data dimension. Importantly, (\ref{eq:thrm1}) holds uniformly  over all subsets 
$\cS$ of size at least $n-k$, meaning that the removed samples may be chosen adversarially. 
Moreover, Theorem \ref{thrm1} directly implies that the robustness measure $\Delta_k(\bv)$ defined in (\ref{eq:delta_v}) converges to zero uniformly in $\bv$. This is because by  its definition, 
\[
\Delta_k(\bv) \leq \max_{\cS \subseteq [n], |\cS| \geq n-k}  \| \widehat \bbeta - \widehat \bbeta_\cS \|  
\]

\noindent Hence, under the conditions of Theorem \ref{thrm1},
\[ \sup_{\bv \in \mathbb{S}^{p-1}} \Delta_k(\bv) \xrightarrow{\,p\,} 0 .\] 

In summary, when the training data is “well-behaved”---that is, the samples are i.i.d.\ from a light-tailed distribution with a well-conditioned covariance matrix $\Sigma$---OLS is robust to $k \ll n$ removals. Hence, if the removal of a small subset of samples significantly alters the fitted coefficients,
careful investigation is warranted. Potential causes include among others, non-i.i.d.\ data, heavy-tailed data, or data with an ill-conditioned covariance matrix.

Theorem \ref{thrm1} is an immediate consequence of the following non-asymptotic bound:

\begin{theorem}\label{lem1:conc_ineq_for_thrm1} Let $(\bx_1, y_1), \ldots, (\bx_n, y_n) $ be  i.i.d.\ samples from Model \ref{mod1} and define 
\begin{align}  
& \omega \coloneqq \|\Sigma^{-1/2}\bx\|_{\psi_2} , & \eta \coloneqq (1 + \omega^3) \|\Sigma^{-1/2}\|  \|y\|_{\psi_2} .
    \label{eq:eta_def}
\end{align}
Assume $k \leq n/2$ and  
\begin{align} \label{thrm3.2cond} \omega^2 \bigg(\sqrt{\frac{k}{n} \log\Big( \frac{en}{k} \Big) } + \sqrt{\frac{p}{n}} \hspace{.15em} \bigg) \leq c.
\end{align} 
Then, for any $t > 0$, with probability at least  $1 - 3(en/k)^{-kt^2} - 4e^{-cn/\omega^4}$,
\begin{equation}
    \begin{aligned} \label{f2}
 &   \max_{\cS \subseteq [n], |\cS|\geq n-k}  \|\widehat \bbeta - \widehat \bbeta_\cS\| \leq    C \eta  
    \bigg( \frac{k}{n}\log \Big(\frac{en}{k}\Big) + \frac{1}{n}\sqrt{k p \log \Big(\frac{en}{k}\Big)} \hspace{.1em}\bigg) (1+t)^2 . 
    \end{aligned}
\end{equation}
\end{theorem}

{

\begin{remark} \label{rem2026}
    The sub-Gaussian condition (i) in Model~\ref{mod1} is essential to Theorem~\ref{lem1:conc_ineq_for_thrm1}. The proof of (\ref{f2}) relies on concentration of quantities such as $\sigma_p(Z_\cS)$ and $\|Z_\cS^\top y_\cS\|$, uniformly over all subsets $\cS\subseteq[n]$ with $|\cS|\ge n-k$. With heavy-tailed covariates or responses, this may fail:  the exclusion of a few extreme observations may strongly affect the OLS estimator. 
    Thus, under heavy-tailed data, OLS may be non-robust to sample removals even in Region~I of Table~\ref{tab:robustness}, where $(p+k)/n\to0$.
\end{remark}

}

{\begin{remark} \label{rem2026a}
The dependence on $n$, $p$, and $k$ in \eqref{f2} is natural and already arises in the simpler setting of multivariate mean estimation. Specifically, suppose that $\bx_1,\ldots,\bx_n$ are i.i.d.\ realizations of a sub-Gaussian vector $\bx$ with unknown mean $\boldsymbol{\mu}$.
Then an argument similar to the proof of Theorem~\ref{lem1:conc_ineq_for_thrm1} yields that, with probability at least $1-e^{-t}$, 
\begin{align}
\sup_{\cS \subseteq [n]: |\cS|\ge n-k}
\bigg\| \frac{1}{n} \sum_{i=1}^n \bx_i - \frac{1}{|\cS|}\sum_{i \in \cS} \bx_i \bigg\|
\leq 
C \|\bx - \boldsymbol{\mu}\|_{\psi_2} 
\left( \frac{k}{n} \sqrt{\log\Big( \frac{en}{k} \Big)}
+ \frac{1}{n} \sqrt{k(p+t)}
\right) , 
\label{2026asdf} 
\end{align}
where $C > 0$ is an absolute constant. 
Thus, the sample mean is robust to sample removals, with dependence on $n$, $p$, and $k$ matching \eqref{f2} up to logarithmic factors.
\end{remark}

\begin{remark}
    \label{rem:ub_delta_k(V)}
In the setting of Theorem~\ref{lem1:conc_ineq_for_thrm1}, for a fixed direction $\bv \in \mathbb{S}^{p-1}$, the robustness measure $\Delta_k(\bv)$
satisfies a bound analogous to \eqref{f2},   without the second term (which arises from uniformly bounding all directions).  In particular, an argument similar to the proof of Theorem~\ref{lem1:conc_ineq_for_thrm1} yields that, with high probability,
\begin{align}
    \label{eq:ub}
      \Delta_k(\bv)
  \leq
  C \eta \cdot  \frac{k}{n}\log\frac{en}{k}.
\end{align}
In Appendix~\ref{appendix:diabetes}, we illustrate this scaling empirically on the benchmark Diabetes dataset from \cite{efron2004least}, whose features and response appear light-tailed.
\end{remark}
}

 We conjecture that the bound in (\ref{f2}) is optimal up to logarithmic factors. In Theorem \ref{thm:asymp_small_p_large_k}, we  establish a lower bound on $\Delta_k(\bv)$ that exhibits terms similar to those appearing in (\ref{f2}). Further discussion and details are provided in Section~\ref{sec:mod2}.

Theorem \ref{lem1:conc_ineq_for_thrm1} bounds the deviation of $\widehat\bbeta_\cS$
from $\widehat\bbeta$. 
To study the deviation of $\widehat \bbeta_\cS$ from the best linear predictor $\bbeta$, defined in \eqref{eq:beta_def}, we now present a bound on $\|\widehat\bbeta- \bbeta \| $. As noted in Section \ref{sec:pb}, bounds on this quantity have been extensively studied. However, as we could not find this precise result in the literature,  we provide its proof in the appendix for completeness.

\begin{theorem}\label{thrm:consistency} Let $(\bx_1, y_1), \ldots, (\bx_n, y_n) $ be  i.i.d.\ samples from Model \ref{mod1}.  Define \begin{align*} 
 &  \omega \coloneqq \|\Sigma^{-1/2}\bx\|_{\psi_2}, &
 \widetilde \eta \coloneqq \big(1+\sqrt{\kappa(\Sigma)} \big) \big\|\Sigma^{-1}(y \bx - \E[y \bx]) \big\|_{\psi_1}   + \sqrt{\kappa(\Sigma)} \omega^2(1+\omega^2)  \|\bbeta\| ,
\end{align*}
and assume $\omega^2\sqrt{p/n} \leq c$. For any $t > 0$, with probability at least $1 - 3e^{-pt^2} - 3e^{-cn/\omega^4}$, 
\begin{equation}
    \begin{aligned} 
    \|  \widehat \bbeta - \bbeta\| \leq  C  \hspace{.08em} \widetilde \eta \hspace{.04em}  \bigg( \sqrt{\frac{p}{n}}(1+t) + \frac{p}{n}(1+ t)^2 \bigg).
    \end{aligned}
\end{equation}
\end{theorem}

\begin{remark}
    Lemmas 2.7.6 and 2.7.7 of \cite{Vershynin12} imply that the sub-exponential norm $ \big\|  \Sigma^{-1} (y\bx - \E[y  \bx]) \big\|_{\psi_1}$, which appears in the definition of $\widetilde \eta$, is bounded by a constant multiple of $\|\Sigma^{-1/2}\|(\|\bx\|_{\psi_2}^2 + \|y\|_{\psi_2}^2)$.
\end{remark}

As $\|\widehat\bbeta_\cS-\bbeta\| \leq \|\widehat\bbeta_\cS - \widehat\bbeta\| + \| 
\widehat\bbeta - \bbeta
\| $,  Theorems \ref{thrm1} and \ref{thrm:consistency} imply the following corollary:
\begin{corollary} \label{cor1} Assume the setting of Theorem \ref{thrm1} and that $\kappa(\Sigma)$ and $\|\bbeta\|$ are bounded.  As $n \rightarrow \infty$ and $(p+k)/n \rightarrow 0$,      \[
\max_{\cS \subseteq [n], |\cS| \geq n-k}   \|\widehat \bbeta_\cS -\bbeta\| \xrightarrow{\,p\,} 0 .\] 
\end{corollary}

Corollary \ref{cor1} guarantees the
consistency of OLS in estimating
the coefficients $\bbeta$ of the best linear predictor, provided that $p =o(n)$ and at most  $k = o(n)$ samples are removed. 
Finally, a natural question is: how many samples can be removed while ensuring that 
$\max_{\cS \subseteq [n], |\cS| \geq n-k} \|\widehat \bbeta_\cS - \bbeta \|$ retains
the error rate of OLS on the full data? 
Comparing the bounds in 
Theorems \ref{lem1:conc_ineq_for_thrm1} and \ref{thrm:consistency},
a sufficient condition is that $    k \ll \sqrt{np}/(\log n)$.
%

\subsection{Results for Model \ref{mod2}}
\label{sec:mod2}

Next, we analyze OLS under Model \ref{mod2}, where the optimal predictor is linear in $\bx$. In this setting, assuming $\varepsilon$ is normally distributed, we derive a non-asymptotic result with explicit constants: 

    
\begin{theorem} \label{thrm:linmodel_nonasymp}
 Let $(\bx_1, y_1), \ldots, (\bx_n, y_n) $ be  i.i.d.\ samples from Model \ref{mod2}. Assume $\varepsilon \sim \mathcal{N}(0,\sigma_\eps^2)$, $k \leq n/2$, and $p \leq n-k$. Then, for any $t, \delta > 0$ such that 
 \begin{align} \label{rho_def} \rho \coloneqq \sqrt{\frac{3k}{n} \log\Big(\frac{en}{k}\Big)} + \sqrt{\frac{p}{n}} + \delta < 1,
 \end{align}
 with probability at least   $1 -  6(en/k)^{-kt^2} -  e^{-n/2} - 2e^{-n\delta^2/2}$,
\comment{\begin{align} \label{f}
  \max_{\cS \subseteq [n], |\cS| \geq n-k} \| \widehat \bbeta  - \widehat \bbeta_\cS \|  \leq  \frac{\|\Sigma^{-1/2}\| \hspace{.1em} \sigma_\eps}{(1-\rho)^4} \cdot \bigg[& \bigg(\frac{6k}{n} \log\Big( \frac{en}{k} \Big) + \frac{2}{n} \sqrt{kp \log\Big( \frac{en}{k} \Big)} \hspace{.15em} \bigg) \bigg(1 + 5 \sqrt{\frac{p}{n}} \hspace{.15em}\bigg)  \nonumber  \\ 
      & +  32\bigg( \frac{k}{n}  \log\Big( \frac{en}{k} \Big) \bigg)^{3/2}(1+t) \bigg] (1+t)^2.
\end{align}}
\begin{align} \label{f}
  \max_{\cS \subseteq [n], |\cS| \geq n-k} \| \widehat \bbeta  - \widehat \bbeta_\cS \|  \leq  \frac{\|\Sigma^{-1/2}\| \hspace{.1em} \sigma_\eps}{(1-\rho)^4} \cdot \bigg[& \bigg(\frac{6k}{n} \log\Big( \frac{en}{k} \Big) + \frac{2}{n} \sqrt{kp \log\Big( \frac{en}{k} \Big)} \hspace{.15em} \bigg) (1+t)^2  \nonumber  \\ 
      & +   \frac{72 \sqrt{k}(k+p)}{n^{3/2}}  \log^{3/2}\Big( \frac{en}{k} \Big)(1+t)^3 \bigg].
\end{align}
\end{theorem}

The first term  on the right-hand side of (\ref{f}) matches (\ref{f2}) in rate; the second term is at most a constant multiple of the first and is negligible when $k \ll n$.
The constants in (\ref{f}) are sharper than those obtained by specializing the proof of Theorem \ref{lem1:conc_ineq_for_thrm1} to Model \ref{mod2}. This improvement comes from using Gaussian comparison inequalities in place of union bounds over subsets of $n-k$ samples. Intuitively, these subsets are highly correlated, making union bounds overly conservative.

For example, a key step in the proof is obtaining a uniform lower bound on $\sigma_p(Z_\cS)$. In Appendix \ref{seca3}, we show that for $t > 0$, with probability at least $1 - e^{-t^2/2}$,
\begin{align} \label{diakvfdnhinhufiihu}
     \min_{\cS \subseteq [n], |\cS| \geq n-k} \sigma_p(Z_\cS) & \geq  \sqrt{n} - \sqrt{p} -  \sqrt{3k \log \Big(\frac{en}{k}\Big)}-t .
    \end{align} 
In comparison, Theorem 4.6.1 of \cite{wainwright} (restated as Lemma \ref{Lem:maxeig2} in the Appendix) combined with a union bound yields, under the same probability guarantee,
\[ \min_{\cS \subseteq [n], |\cS| \geq n-k} \sigma_p(Z_\cS)  \geq  \sqrt{n-k} - \sqrt{p} - 2 \sqrt{k \log \Big(\frac{en}{k}\Big)} - t .\]
Our bound is a strict improvement, with a notable gain when $k \asymp n$.
Proposition 3.3 of \cite{diakonikolas2023algorithmic} gives a similar bound to (\ref{diakvfdnhinhufiihu}), without explicit constants.

Our final result is a lower bound on \( \Delta_k(\bv) \). Since \( \Delta_k(\bv) \rightarrow 0 \) as \( k/n \rightarrow 0 \) (Theorem \ref{thrm1}), we focus on the case where \( k \) is proportional to \( n \). In this regime, OLS is provably non-robust.
\begin{theorem}\label{thm:asymp_small_p_large_k}
    Let $(\bx_1, y_1), \ldots, (\bx_n, y_n) $ be i.i.d.\ samples from Model \ref{mod2}. Assume that as $n \rightarrow \infty$, $k/n \rightarrow \alpha \in (0,1/2)$ and $ p/(n-k) \rightarrow \gamma\in[0,1)$.
    Then, 
        \begin{align}
        \label{eq:asymp_small_p_large_k}
    \lim_{n \rightarrow \infty} \P \bigg( \Delta_k(\bv) \geq   \frac{1}{1-\alpha} \cdot \big\|\Sigma^{-1/2}\bv\big\| \cdot \E\big[ \varepsilon z \,\ind(\varepsilon z>q_{1-\alpha})\big] \bigg) = 1 ,
    \end{align}
    where $z\sim \mathcal{N}(0,1)$ is independent of $\varepsilon$, and $q_{1-\alpha}$ is the $(1-\alpha)$-quantile of $\varepsilon z$.
\end{theorem}
\noindent This result is an immediate corollary of the non-asymptotic bound in Theorem~\ref{thrm:small_p_large_k} in Appendix~\ref{appendix:LB}.

\begin{remark}
    The term $\E\big[\varepsilon z\,\ind(\varepsilon z>q_{1-\alpha})\big]$ in \eqref{eq:asymp_small_p_large_k} 
also appears in \citet{broderick2020automatic} (p.~16). 
However, in their analysis this term arises from a first-order 
Taylor approximation of the effect of removing samples, rather than a lower bound 
on the exact change $\Delta_k(\bv)$ induced in $\widehat \bbeta$. 
Their result assumes $p=1$ and $\varepsilon$ is Gaussian, whereas Theorem~\ref{thm:asymp_small_p_large_k} applies to $p \geq 1$ and sub-Gaussian noise.
\end{remark}

To gain insight into Theorem \ref{thm:asymp_small_p_large_k}, consider the simplified setting where  $\eps\sim \mathcal{N}(0,1)$.
 By Corollary 2.6 in \cite{gaunt2025asymptotic}, 
\begin{align}
    \E[\eps z \ind(\eps  z > q_{1-\alpha})] =
    \alpha \log \left(\frac{1}{\alpha}\right) (1+o_\alpha(1)) .  \nonumber
\end{align}
Substituting this into \eqref{eq:asymp_small_p_large_k}, we obtain that, with probability tending to one,
\begin{align}
\label{eq:asymp_small_p_large_k2}
\Delta_k(\bv) \geq \big\|\Sigma^{-1/2}\bv\big\| \cdot \frac{\alpha}{1-\alpha}\log\left(\frac{1}{\alpha}\right) (1+o_\alpha(1)).
\end{align}
Thus, when \(k \ll n\), \(\Delta_k(\bv)\) scales as 
$\alpha \log(1/\alpha) \asymp k/n \cdot \log(n/k)$, increasing approximately linearly with the number of removals. In this regime, heuristically making the approximation $(X^\top \spa X/n)^{-1} \approx (X_\cS^\top X_{\cS}/n)^{-1} \approx \Sigma^{-1}$,
\begin{align*}
\widehat \bbeta - \widehat \bbeta_\cS = (X^\top X)^{-1} X^\top \by - \big(X_\cS^\top X_\cS\big)^{-1} X_\cS^\top \by_\cS 
& \approx \frac{1}{n} \Sigma^{-1/2} \sum_{i \in \cS^c} \eps_i x_i .
\end{align*}
 The factor of $\log(n/k)$ in $\Delta_k(\bv)$ stems from the extreme values of \(\{\varepsilon_i x_{i1}\}_{i \in [n]}\), which are independent sub-exponential random variables. 
When $p \ll k \ll n$, this matches the order of magnitude of the upper bound given in (\ref{f}). 

We conjecture that the lower bound in \eqref{eq:asymp_small_p_large_k} is sharp, in the sense that under the assumptions of Theorem \ref{thm:asymp_small_p_large_k}, 
\[
\Delta_k(\bv) \xrightarrow{\,\,p\,\,}   \frac{1}{1-\alpha} \cdot \big\|\Sigma^{-1/2}\bv\big\| \cdot \E\big[ \varepsilon z \,\ind(\varepsilon z>q_{1-\alpha})\big] .
\]
We complement this conjecture by an empirical simulation. We generate $N=500$ i.i.d.\ datasets, each consisting of $n=1000$ i.i.d.\ samples drawn from Model \ref{mod2}, with 
 $p \in \{1, 200\}$
and $\eps \sim \mathcal{N}(0,1)$. For a range of values of the ratio $\alpha = k/n$, we compute the following  quantities:
(1) a lower bound on $\Delta_k(\bv)$ using the AMIP algorithm (\cite{broderick2020automatic}), and 
(2) an upper bound on $\Delta_k(\bv)$ using the ACRE algorithm (\cite{hopkins}).\footnote{See Section \ref{sec:ACRE} for details on ACRE.} The AMIP and ACRE bounds are data-dependent and non-probabilistic.

In Figure \ref{fig:deltak_comp_p_1},
we compare the averages of the AMIP and ACRE bounds over the $N=500$ datasets to our theoretical lower bound in \eqref{eq:asymp_small_p_large_k}. 
\comment{
The left and right panels of  Figure \ref{fig:deltak_comp_combined}  correspond to $p=1$ and  $p = 200$, respectively. 
}
For small values of $\alpha$, all bounds closely align, suggesting that the lower bound in \eqref{eq:asymp_small_p_large_k} is fairly tight.
In particular, for 
 $p= 1$ and 
$\alpha \leq .05$, the ratio of the ACRE upper bound to our lower bound is at most $1.3$.
Note that we do not plot the ACRE upper bound for $p = 200$. In moderate to high dimensions, for $\alpha \gtrsim .03$, the ACRE upper bound is quite loose and, if plotted, would dominate the scale of the figure and obscure the lower bounds.

\begin{figure}[t]
    \centering
    \includegraphics[width=0.49\linewidth]{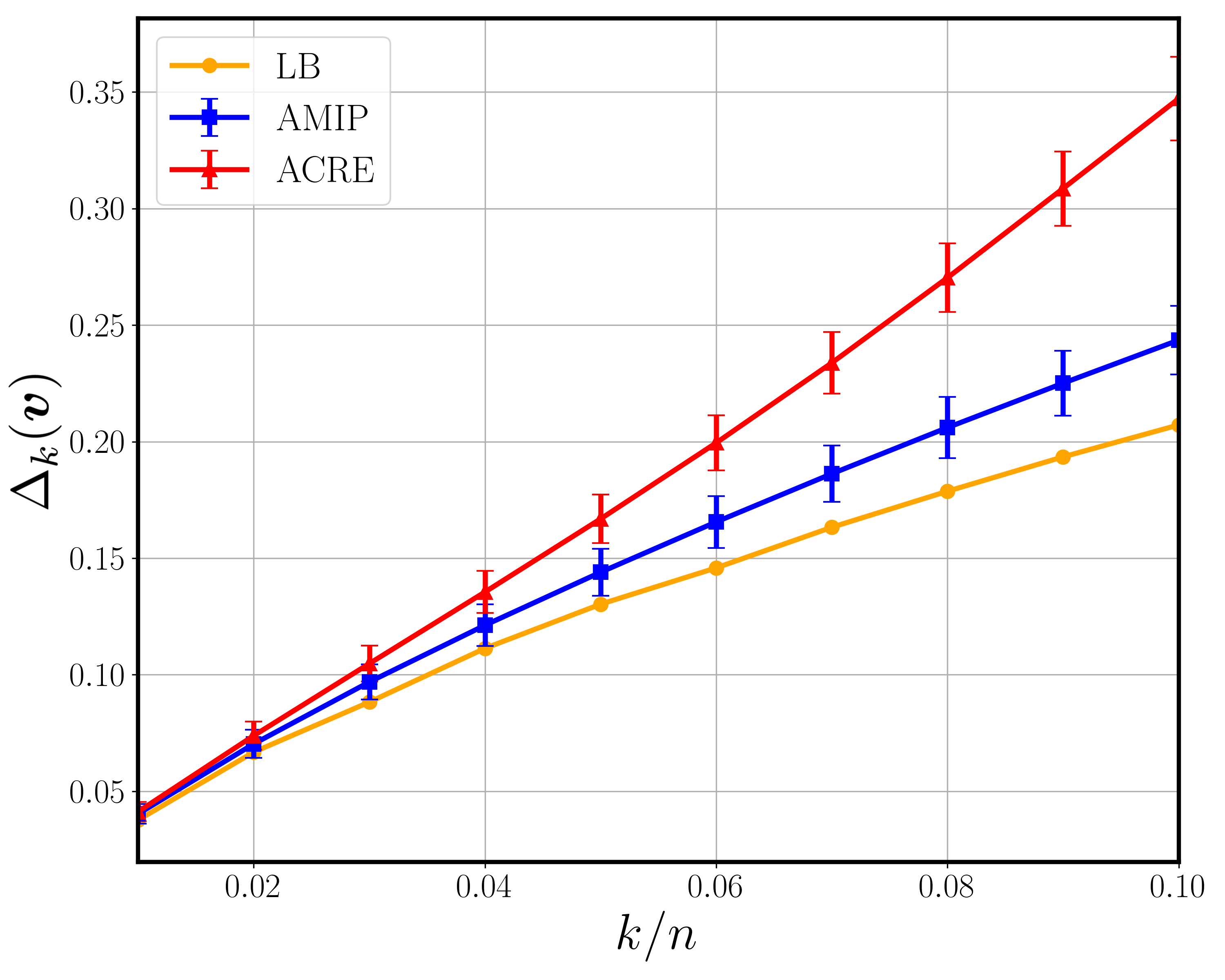}
    \includegraphics[width=0.49\linewidth]{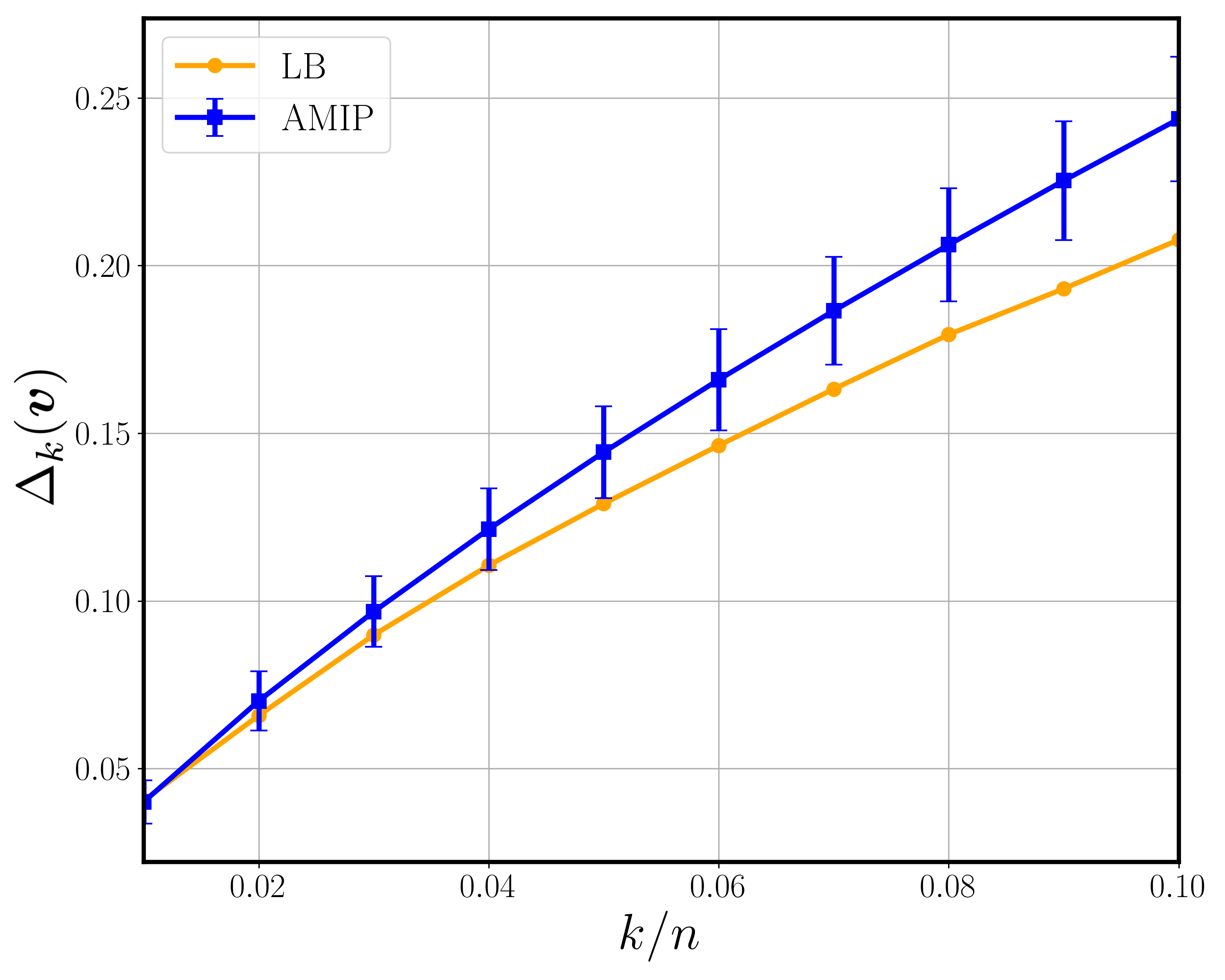} 
    \caption{Comparison of bounds on ${\Delta}_k(\bv)$ with $p=1$ (left) and $p=200$ (right). The orange curve  denotes our theoretical lower bound in \eqref{eq:asymp_small_p_large_k}, the blue curve is an empirical lower bound  computed using AMIP \citep{broderick2020automatic}, and the red curve is an empirical upper bound  computed using ACRE \citep{hopkins}. Both empirical curves are averaged over multiple i.i.d.\ datasets generated from Model~\ref{mod2}, with error bars indicating one standard deviation.  See Section \ref{sec:mod2} for details of the simulation setup.}
    \label{fig:deltak_comp_p_1}
\end{figure}

\subsection{Theoretical Guarantees for ACRE}\label{sec:ACRE}

Recently, \cite{hopkins} introduced 
a robustness auditing method called 
ACRE (Algorithm for Certifying Robustness Efficiently). 
For a fixed direction $\bv$,  ACRE computes upper and lower bounds $U_{k}(\bv)$ and $L_k(\bv)$ such that 
without any modeling assumptions, 
\[ L_k(\bv) \leq \Delta_k(\bv) \leq U_k(\bv) . \] 

\noindent Under  a variant of Model \ref{mod2} with slightly weaker tail conditions, \cite{hopkins} proved\footnote{The left-hand bound in (\ref{acre}) follows from the proof of Claim 7.2 in \cite{hopkins}, which assumes the normalization $X^\top X = I_p$. Translating to our setting introduces a factor of $\sqrt{n}$ in the denominator relative to the penultimate displayed equation in their proof.
}
that
 there exists a threshold $K = \widetilde O\big( \min(n/\sqrt{p}, n^2/p^2)\big)$
 such that for any $k \leq K$, with high probability,
\begin{align} \label{acre}
& U_k(\bv) = \widetilde O \bigg(\frac{k}{n} + \frac{kp + k^2\sqrt{p}}{n^{2}} \bigg) , & \frac{U_k(\bv)}{L_k(\bv)} = 1 + \widetilde O\bigg(\frac{p+k\sqrt{p}}{n} \bigg)  . 
\end{align}
 Here,  $\widetilde O$ suppresses polylogarithmic factors in $n$. In contrast to (\ref{acre}), Theorem \ref{lem1:conc_ineq_for_thrm1} (1) pertains to Model \ref{mod1} rather than a linear model, (2) permits $k \leq n/2$, and (3) holds for all directions $\bv \in \mathbb{S}^{p-1}$ simultaneously (that is, we bound $\sup_{\bv \in \cS^{p-1}} \Delta_k(\bv)$).  The rate in Theorem \ref{lem1:conc_ineq_for_thrm1} matches the bound on $U_k(\bv)$ in (\ref{acre}) under either of the following conditions: (1) $k \gtrsim p$, or (2) $k \ll p$ and $n \lesssim \sqrt{kp} + k^{3/2}$. The rate is superior under (1) $k \gtrsim p$ and $n \ll k\sqrt{p}$, or (2) $k \ll p$ and $n \ll \sqrt{kp} + k^{3/2}$. Note that efficiently computable bounds on $\Delta_k(\bv)$, such as $U_k(\bv)$, are not expected to be rate-optimal.
 
Rubinstein and Hopkins interpret (\ref{acre})  to mean that when $p+k\sqrt{p} \ll n$, the upper and lower bounds $U_k(\bv)$ and $L_k(\bv)$ are tight and provide a reasonable proxy for the degree of robustness. Our theoretical results demonstrate that OLS is {\it provably} robust to removals in a broader parameter regime characterized by $p + k \ll n$.  
An interesting open question is whether the upper and lower bounds of ACRE remain tight in more general regimes, in particular where OLS is non-robust.


\section{Robustness of Cash Transfers Data} \label{sec:exp}

\begin{table*}[t]
    \caption{
    Statistics of the six Cash Transfers datasets, corresponding to three time periods and two groups (poor/non poor). 
    Column 6 (\texttt{1-Greedy})  is the 
    size of the smallest subset identified by \texttt{1-Greedy}
    whose removal changes the sign of $\widehat\beta_1$.
    Columns 7 ($\mu_y$) and 8 ($\sigma_y$) are the empirical mean and standard deviation of the response $y$. 
    Columns 9 ($>\hspace{-.2em}5\sigma_y$) and 10 ($>\hspace{-.2em}10\sigma_y$) are the number of samples that deviate from $\mu_y$ by at least $5 \sigma_y$ and $10\sigma_y$.
    Columns 11 ($\mu_{y}^{{\texttt{1-G}}}$) and 12 ($y_{\max}^{\texttt{1-G}}$) are the average and maximum $y$-values of the subset found by \texttt{1-Greedy}.
  }
  \label{tab:cash_transfers_stat}
  \centering
  \begin{tabular}{cccccccccccc}
    \toprule
    Dataset & Period & Poor & $n$ & $\widehat\beta_1$ & \texttt{1-Greedy} & $\mu_y$ & $\sigma_y$ & $ > \hspace{-.2em}5\sigma_y$ &  $ > \hspace{-.2em}10\sigma_y$  &  $\mu_{y}^{\texttt{1-G}}$ & $y_{\max}^{\texttt{1-G}}$ \\
    \midrule
    1 & 8  & Y & 10781 & 16.53 & 224 & 170 & 126 & 48 & 12 &  584  & 4380 \\
    2 & 8  & N & 4543  & -5.53 & 5   & 219 & 172 & 29 & 5  &  2018 & 2483 \\
    3 & 9  & Y & 9489  & 28.65 & 314 & 176 & 182 & 48 & 15 &  622  & 5117 \\
    4 & 9  & N & 3769  & 23.19 & 21  & 226 & 273 & 20 & 9  &  2670 & 5801 \\
    5 & 10 & Y & 10368 & 32.52 & 555 & 172 & 156 & 56 & 13 &  450  & 5080 \\
    6 & 10 & N & 4191  & 21.12 & 26  & 217 & 267 & 19 & 7  &  2154 & 7470 \\
    \bottomrule
  \end{tabular}
\end{table*}

In this section, we revisit prior analyses of OLS robustness in the context of a widely cited econometric study on cash transfers by \citet{Angelucci09}.
A corresponding robustness analysis of the Nightlights dataset \citep{martinez2022much} appears in the Appendix.
The Cash Transfers data was retrieved from the repository provided by \citet{hopkins}.\footnote{
https://github.com/ittai-rubinstein/ols\_robustness
} We follow the same data pre-processing steps as in Appendix C.3 of \cite{hopkins}, including the removal of a small number of specific observations.


 \citet{Angelucci09} investigated the economic impact  of Mexico's Progresa aid program on 506 rural villages. Among these villages, 320 were randomly selected to participate in the program, while the remaining 186  served as a control group.  
In participating villages, financial aid  was given via cash transfers to eligible poor households. This also had an indirect effect on non-poor households in the same villages.

The authors estimated the effects of the program on both poor and non-poor households by a linear regression approach. 
%
Specifically, they constructed six separate datasets based on the household status (poor or non-poor) and three distinct time periods.
For each dataset, they fitted the following linear  model:
\[
    y = \beta_0 + \beta_1 x_1 + \sum_{j=2}^{17} \beta_j x_j,
\]
where $y$ is the total household consumption in pesos, $x_1$  is a binary treatment variable (indicating whether the corresponding village participated in Progresa), and $x_2, \ldots, x_{17}$  are additional covariates of each household.
Hence, the coefficient $\beta_1$ captures the effect of the Progresa aid program. 
Table \ref{tab:cash_transfers_stat}
provides a summary of these datasets.
Column 5 gives the estimated treatment effect.  As expected, it is larger for poor households (directly affected by Progresa) than for non-poor households (indirectly affected).

\citet{broderick2020automatic}, \citet{hopkins}, and \citet{kuschnig2021hidden} showed that, for three of the six datasets, removing a small number of observations substantially alters the estimated OLS coefficients. 
Column~6 in Table \ref{tab:huber} reports the size of the minimal subset whose removal changes the sign of the estimated treatment effect, as identified by the auditing method proposed by \citet{kuschnig2021hidden}, which we term \texttt{1-Greedy}. 
For example, in dataset~2, just five removals (\(0.11\%\)) change the sign of \(\widehat{\beta}_1\), 
and in dataset~4, 21 removals (\(0.55\%\)) suffice.


\citet{huang2024approximations} showed empirically that \texttt{1-Greedy} provides the tightest upper bound on the minimum number of removals required to change the sign of \(\widehat \beta_1\), compared to AMIP and ACRE. The algorithm operates as follows. At step \(k\), it iterates over the \(n - k\) remaining observations, removes each one individually, and computes the corresponding regression coefficients. It then removes the sample whose deletion induces the largest decrease in magnitude in the regression coefficient of interest. This procedure continues until the coefficient changes sign.

\begin{table*}[t]
  \caption{
    Comparison of OLS and Huber regression under sample removal stress tests using the \texttt{1-Greedy} algorithm.
    Each row corresponds to a different dataset.
    Columns $\widehat\beta_1$ and $\widehat \beta_1^{\text{Huber}}$ report the estimated treatment effect under standard OLS and Huber regression, respectively.
    Columns “\texttt{1-Greedy} {OLS}” and “\texttt{1-Greedy} Huber” give the size of the smallest subset identified by \texttt{1-Greedy} whose removal reverses the sign of $\widehat\beta_1$ and $\widehat\beta_1^{\text{Huber}}$, respectively.}
  \label{tab:huber}
  \centering
  \begin{tabular}{cccccccc}
    \toprule
    Dataset &Period & Poor & $n$ & $\widehat\beta_1$ & \texttt{1-Greedy} {OLS}  & $\widehat{\beta}_1^{\text{Huber}}$ & \texttt{1-Greedy} Huber \\
    \midrule
     1 & 8  & Y  & 10781 & 16.53 & 224 & 14.16  & 570  \\
     2 & 8  & N  & 4543  & -5.53 & 5   & -0.52  & 5   \\
     3 & 9  & Y  & 9489  & 28.65 & 314 & 22.13  & 817 \\
     4 & 9  & N  & 3769  & 23.19 & 21  & 11.21  & 124 \\
     5 & 10 & Y  & 10368 & 32.52 & 555 & 26     & 1145  \\
     6 & 10 & N  & 4191  & 21.12 & 26  & 13.42  & 162  \\
    \bottomrule
  \end{tabular}
\end{table*}

\subsection{Heavy-tailed Responses and Non-robustness
of OLS}

The empirical findings presented above seem  to contradict our theoretical results on the robustness of OLS to $k \ll n$ removals. 
As we now discuss, this seeming discrepancy is due to the heavy-tailed distribution 
of the response.
Columns 7 and 8 of Table \ref{tab:cash_transfers_stat} report the mean $\mu_y$ and standard deviation $\sigma_y$ of the response $y$ across the six datasets. 
For each household type (poor/non-poor) the means are similar across the three time periods, with consumption in non-poor households roughly 50 pesos higher than in poor ones. 
Crucially, columns 9 and 10 show that each dataset contains a non-negligible number of samples whose response $y$ deviates by at least 5 or even 10 standard deviations from the mean. In simple terms, the response $y$ is quite heavy-tailed.

{
Next, we examine the response values of the few samples selected for removal by \texttt{1-Greedy}. 
As shown in column~11 of Table~\ref{tab:cash_transfers_stat}, these samples have \emph{extremely large} \(y\)-values, with many exceeding \(\mu_y\) by more than 5 standard deviations. 
Moreover, column~12 shows that the maximal selected \(y\) value is between 13 and 33 standard deviations above \(\mu_y\)! 
It is thus unsurprising that these subsets of samples have an outsized effect on the fitted model.
 We remark that in 
dataset 2
---where only 5 removals suffice to reverse the sign of $\widehat \beta_1$---the estimated treatment effect computed on the full data
was to begin with not statistically significant, even at the 10\% level (Table 1 of \citet{Angelucci09}). 
}

\subsection{Robustness under the Huber Loss} \label{subsec:huber_exp}

For datasets with a heavy-tailed response, a standard approach is to fit a regression model using a robust loss function, such as Huber or \(L_1\), rather than ordinary least squares.
In this section, we study the robustness of this approach to sample removals in the six Cash Transfers datasets.

Given a training set $(\bx_1,y_1),\ldots,(\bx_n,y_n)$, assuming a linear model, this method
solves the following optimization problem: 
\begin{equation}
\widehat{\bm{\beta}}^{\text{Huber}} \coloneqq \argmin_{\bb \in \mathbb{R}^p} \sum_{i=1}^n h_\tau (\bb^\top \spa \bx_i - y_i)
    \label{eq:Huber}
\end{equation}
where $h_\tau$ is the Huber loss function, given by 
%
\begin{align}
    \label{huber:} 
    h_\tau(z) \coloneqq
    \begin{dcases}
        \frac{z^2}{2} & | z | \leq \tau, \\
        \tau \Big( | z | - \frac{\tau}{2}\Big) & |z| > \tau . 
    \end{dcases}
\end{align}
{Here, $\tau > 0$ controls the transition from squared loss to absolute loss.}
{
In OLS, observations with larger residuals exert greater influence on the fitted coefficients, and their removal may significantly perturb the model. In contrast, the Huber loss $h_\tau$ clips the derivative of the residual at level $\tau$, limiting the extent to which large-residual observations can dominate the fit and making the estimator more stable to their removal.
}

{
For each of the six datasets, we assessed their robustness as follows: (1) compute $\widehat{\bm{\beta}}^{\text{Huber}}$, the solution of \eqref{eq:Huber} on the full data with $\tau=1$; 
(2) run \texttt{1-Greedy} with the Huber loss to identify a subset of  samples whose removal reverses the sign of $\widehat\beta_1^{\text{Huber}}$.
}

{

As shown in Table~\ref{tab:huber}, 
across all datasets, $\widehat \beta_1$ under Huber regression has the same sign as under OLS, though typically with slightly smaller magnitude. Importantly, with Huber regression, the subsets identified by \texttt{1-Greedy} are considerably larger than for OLS. For instance, in dataset 6, reversing the sign of $\widehat \beta_1$ requires removing 162 samples, compared with just 26 for OLS.
The only exception is the second dataset, where 5 removals suffice to flip the sign of $\widehat \beta_1$.
However, as noted by \citet{Angelucci09}, this coefficient is not statistically significant to begin with.
}

{
We remark, however, that the empirical advantage of Huber regression is tied to the presence of heavy tails or extreme outliers. For ``well-behaved,'' light-tailed data, Huber regression need not provide additional robustness to adversarial sample removals. 
Empirical evaluations under Model \ref{mod2} suggest that for such data, the worst-case sample removal perturbation for Huber regression exhibits the same dependence on $k$ as our theoretical bounds for OLS derived in Remark~\ref{rem:ub_delta_k(V)}.
}

In summary, assuming the \texttt{1-Greedy} approximation is reasonably accurate, in the presence of heavy-tailed data, Huber regression can be substantially more robust than OLS to sample removals. A theoretical investigation of the properties of such estimators under sample removals is an interesting direction for future work.

\nocite{*}
\bibliographystyle{apalike} 
\bibliography{bibliography.bib}

\appendix

\newpage

\section{Robustness Auditing for the Nightlights Data}

This section presents empirical results on the robustness to sample removals for the Nightlights dataset 
\citep{martinez2022much}, and contrasts them with the prior analyses of \cite{martinez2022much} and \citet{hopkins}. 
Martinez constructed the Nightlights data to investigate the potential misreporting of Gross Domestic Product (GDP) in autocratic versus democratic regimes. His approach involved comparing official GDP figures with satellite-based measurements of nighttime lights (NTL), which serve as a widely recognized proxy for economic activity. To assess the level of democracy in each country, Martinez used the \emph{Freedom in the World (FiW)} index, which ranges from 0 to 6, with higher values indicating less democratic regimes.
The constructed dataset covers 184 countries over the 22-year period 1992–2013, for a total of $n=3895$ samples.

Next, the following regression model was fit, 
 where $i$ denotes a country and $t$ a specific year: 
\begin{equation}
\label{eq:martinez}
\ln(\text{GDP})_{i,t} = \mu_i + \delta_t + \beta_0 \ln(\text{NTL}_{i,t}) + \beta_1 \text{FiW}_{i,t} + \beta_2 \text{FiW}^2_{i,t} + \beta_3 \big[ \ln(\text{NTL}) \times \text{FiW} \big]_{i,t}.
\end{equation}
Here, $\text{GDP}_{i,t}$, $\text{NTL}_{i,t}$ 
and $\text{FiW}_{i,t}$ represent, respectively,
the self reported GDP, the average of nightlight satellite measurements, and the FiW index for country $i$ in year $t$. The parameters $\mu_i$ and $\delta_t$ are country and year fixed effects.
The coefficients $\beta_0, \beta_1,$ and $\beta_2$ capture baseline dependencies between GDP, NTL, and the level of political freedom,
while $\beta_3$ is a measure of the overstatement
of GDP in autocratic countries. 
Via ordinary least squares, \citet{martinez2022much} estimated $\widehat\beta_3\approx 0.021$, which corresponds to a 35\% overestimation of GDP growth in autocracies. 

\begin{figure}[t]
    \centering
    \subfloat{
        \includegraphics[width=0.48\textwidth]{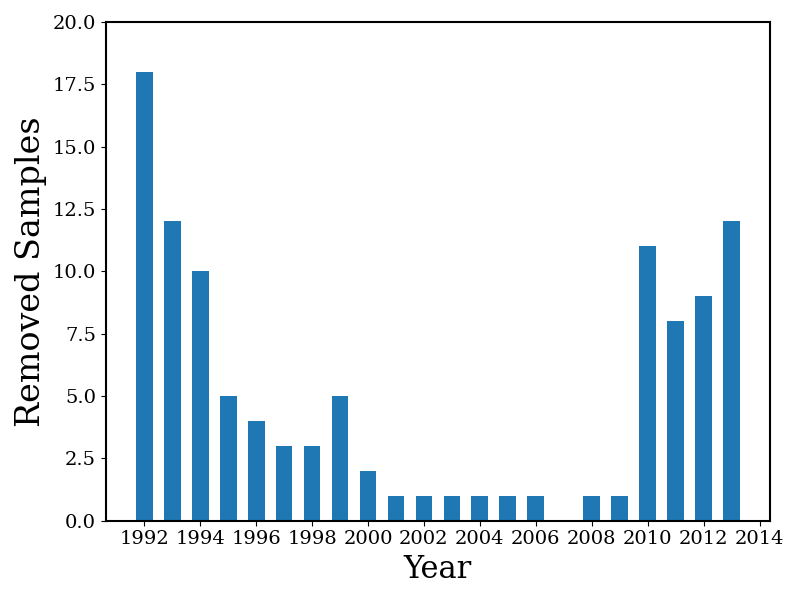}
    }
    \hfill 
    \subfloat{
        \includegraphics[width=0.48\textwidth]{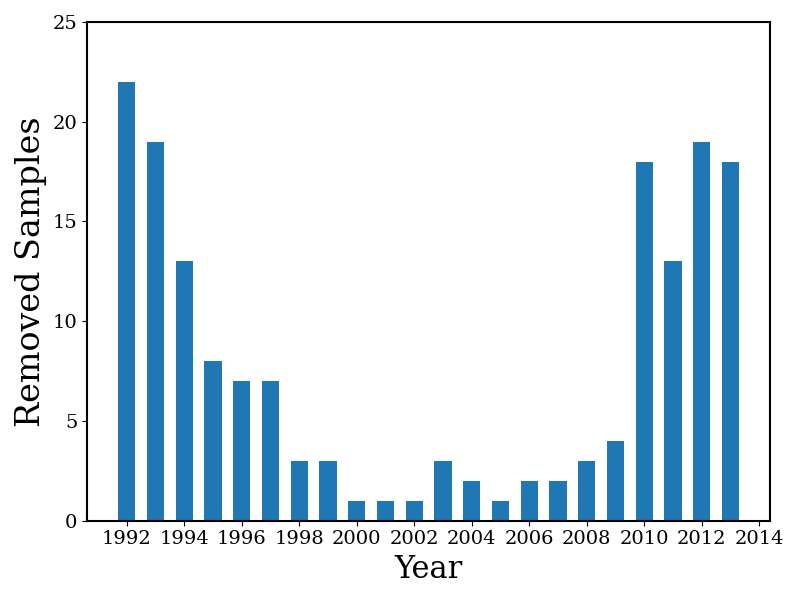}
    }
    \caption{Number of samples per year identified for removal by the \texttt{1-Greedy} method under OLS (left) and Huber regression (right). In both cases, removals are highly non-uniform, with a greater concentration in the early and late years of the sample period.}
    \label{fig:ols_huber_vs_time}
\end{figure}
To assess the stability of his findings, \citet{martinez2022much} applied the AMIP algorithm  \citep{broderick2020automatic}, which
identified a subset of 136 observations whose removal reverses the sign of the estimate of $\beta_3$. 
Applying the  \texttt{1-Greedy} method, 
\citet{hopkins} identified a smaller subset of 110 observations (less than $3\%$ of samples)  whose removal reverses the sign.

Here, we assess the robustness of the regression model \eqref{eq:martinez} using Huber loss in place of the standard $L_2$ loss.
With Huber regression (using $\tau=1$), we obtain
$\widehat{\beta}_3^{\text{Huber}} = 0.0184$, which is slightly smaller than the corresponding OLS estimate.
Importantly, Huber regression is considerably more robust: using the \texttt{1-Greedy} method, 170 samples need to be removed to reverse the sign of the estimate  of $\beta_3$, representing an increase of over $50\%$ compared to the 110 observations required for OLS (as reported by \citet{hopkins}).

It is worth examining the composition of the 170 samples identified for Huber regression, and to compare them with the 110 samples identified for OLS.
Recall that each of the $n=3895$ observations corresponds to a specific (country, year) pair. Consequently, the samples exhibit dependence both over time within a given country and across countries within a given year.
Figure~\ref{fig:ols_huber_vs_time} displays the distribution of the removed samples as a function of year, for both Huber regression and OLS. As expected given the correlation structure of the data, the removals are highly non-uniform in time, with a disproportionate concentration in the earliest and latest years of the study period.
Moreover, as Martinez noted, the removed samples are also highly unevenly distributed across the 184 countries, with only a small number of countries accounting for a large fraction of the removals.

A more nuanced picture emerges by examining the democratic levels of the countries contributing the largest number of removed samples. We focus on countries with at least five samples marked for removal and follow the classification scheme of \citet{martinez2022much}, which divides countries into three categories based on their FiW index: free $(\text{FiW}<2)$, partially free $(2\le \text{FiW}\le 4)$, and not free $(\text{FiW}>4)$.
As shown in Figure~\ref{fig:ols_fiw} (left), under OLS there are eight such countries, with a relatively balanced distribution across democratic categories (three not free, two partially free, and three free).
In contrast, Figure~\ref{fig:ols_fiw} (right) shows that under Huber regression there are thirteen such countries, and their distribution is markedly more skewed toward autocracies (eight not free, three partially free, and two free).
This pattern may merit further investigation beyond the scope of this paper.
%

\begin{figure}[ht]
    \centering
    \includegraphics[width=0.99\linewidth]{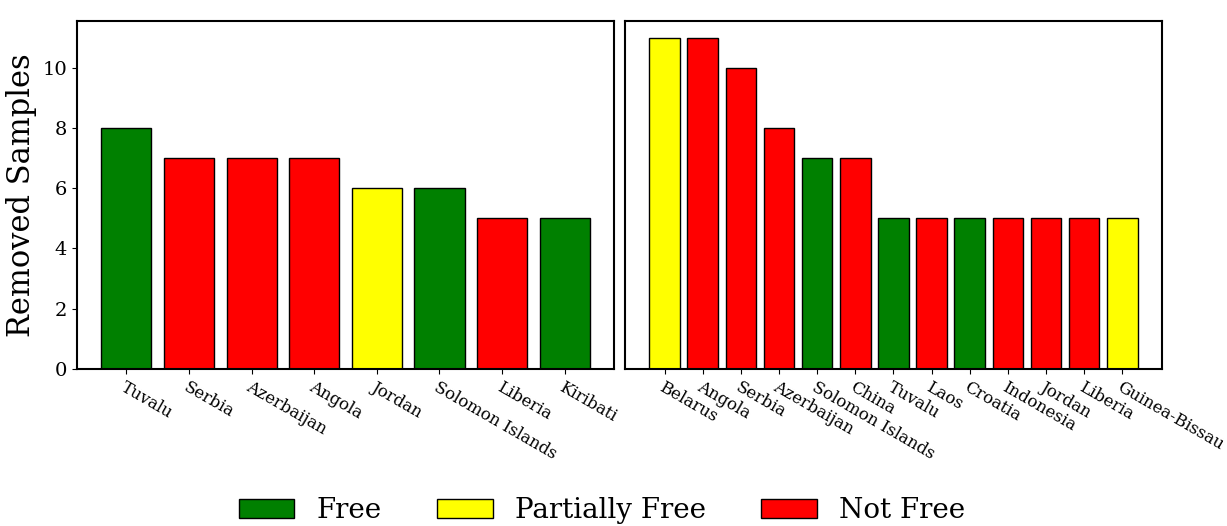}
    \caption{Number of samples removed per country by the \texttt{1-Greedy} OLS. 
    This removal set of 110 samples is sufficient to flip the sign of $\widehat\beta_3$. Countries are classified by their average FiW index (1992–2013).}
\label{fig:ols_fiw}
\end{figure}

{\section{Robustness Auditing for the Diabetes Dataset}
\label{appendix:diabetes}

This section presents an empirical analysis of robustness to sample removals using the well-known Diabetes dataset \citep{efron2004least}, which has become a standard benchmark in statistics and machine learning for evaluating linear regression methods.
The dataset consists of  $n=442$ patients, each described by $p=10$ baseline features: age, sex, body mass index (BMI), average blood pressure, and six blood serum measurements. 
The response variable $y$ is a continuous, quantitative measure of disease progression one year after the baseline measurements. 
As is standard practice for this dataset, the features are mean-centered and standardized.

Unlike the Cash Transfers and Nightlights datasets considered previously, the Diabetes dataset appears relatively ``well-behaved'', 
allowing us to compare the empirical behavior of $\Delta_k(\bv)$ with the prediction of Theorem~\ref{lem1:conc_ineq_for_thrm1}.
Empirically, both the features and the response have light tails; Figures~\ref{fig:histograms} and ~\ref{fig:diabetes_y_greedy} (left) show histograms of nine  features and the response, respectively. The remaining feature, sex, is categorical, with approximately $53\%$ of patients male.
While the features exhibit moderate collinearity, with $\kappa(\Sigma_n) \approx 470$, this affects the prefactor in (\ref{f2}) in Theorem~\ref{lem1:conc_ineq_for_thrm1}, but not the predicted dependence on the number of removed samples $k$.

Following the standard setup for this dataset, we fit the following OLS model:
\begin{align}
 y = \beta_0 + \beta_1 x_1 + \sum_{j=2}^{10} \beta_j x_j,
\end{align}
where $x_1$ denotes BMI and $\{x_j\}_{j=2}^{10}$ denote the remaining features.

Recall that, in the setting of Theorem~\ref{lem1:conc_ineq_for_thrm1}, Remark~\ref{rem:ub_delta_k(V)} states the high-probability bound
\begin{align}
      \Delta_k(\bv)
  \leq
  C \eta \cdot\frac{k}{n}\log\frac{en}{k}. \nonumber
\end{align}

\noindent To empirically validate this rate, we ran the \texttt{1-Greedy} algorithm proposed by \cite{kuschnig2021hidden} and described in Section \ref{sec:exp}. We evaluated this algorithm for $k = 1, 2, \ldots, \lfloor 0.1n \rfloor = 44$, with $\bv = e_1$ (corresponding to BMI).
Figure~\ref{fig:diabetes_y_greedy} (right) compares the \texttt{1-Greedy} lower bound with $C(k/n)\log(en/k)$, where $C$ is chosen for visual alignment. The two curves exhibit similar scaling, consistent with our theoretical prediction that $\Delta_k(\bv)$ scales as $(k/n)\log(en/k)$.

\begin{figure}[htbp]
  \centering
  \includegraphics[width=0.6\textwidth]{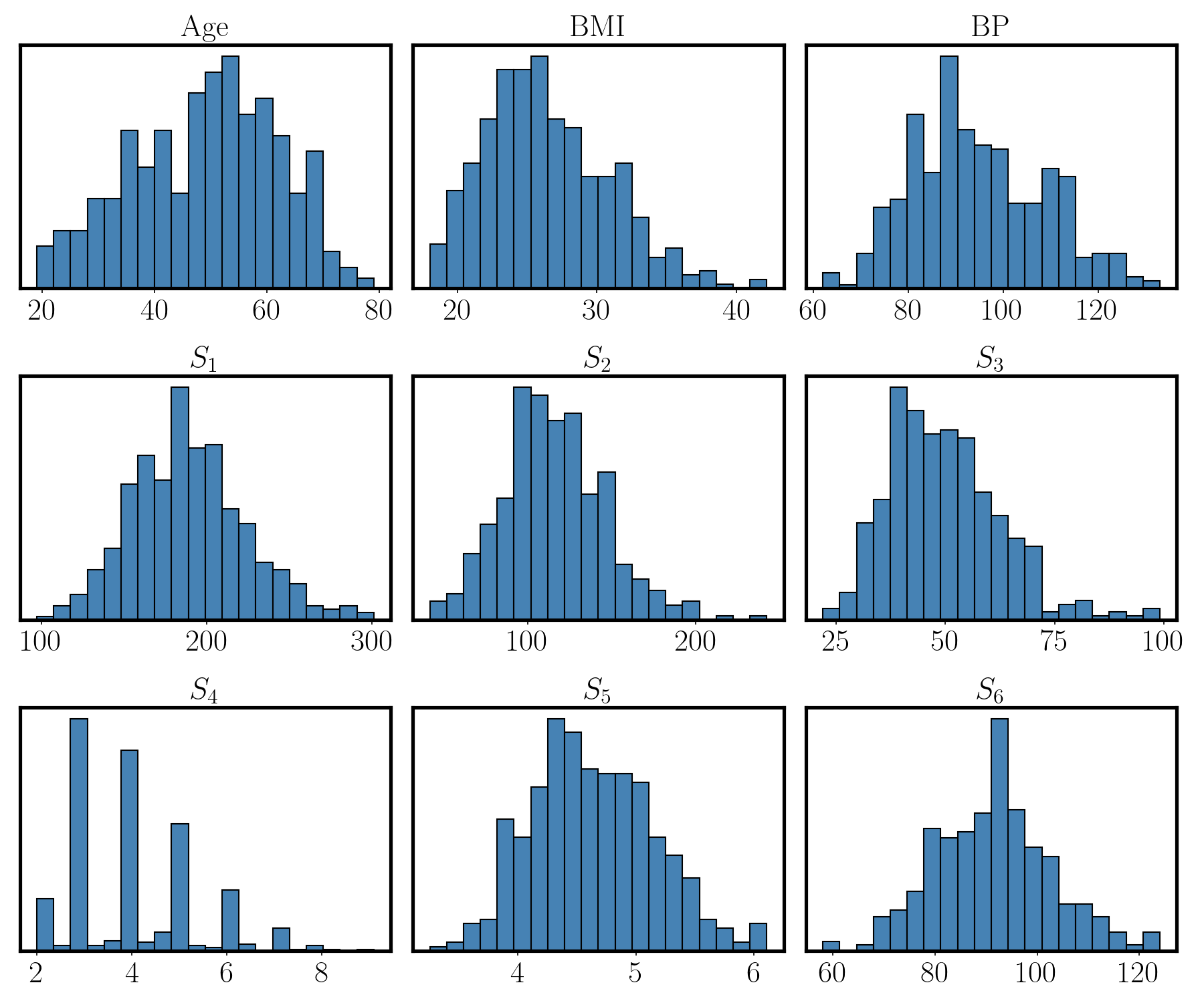}
  \caption{Histograms of the 9 continuous features in the Diabetes dataset \citep{efron2004least}, shown in raw unstandardized units.
    }
  \label{fig:histograms}
\end{figure}

\begin{figure}[!t]
  \centering
  \raisebox{0.2\height}{\includegraphics[width=0.48\textwidth]{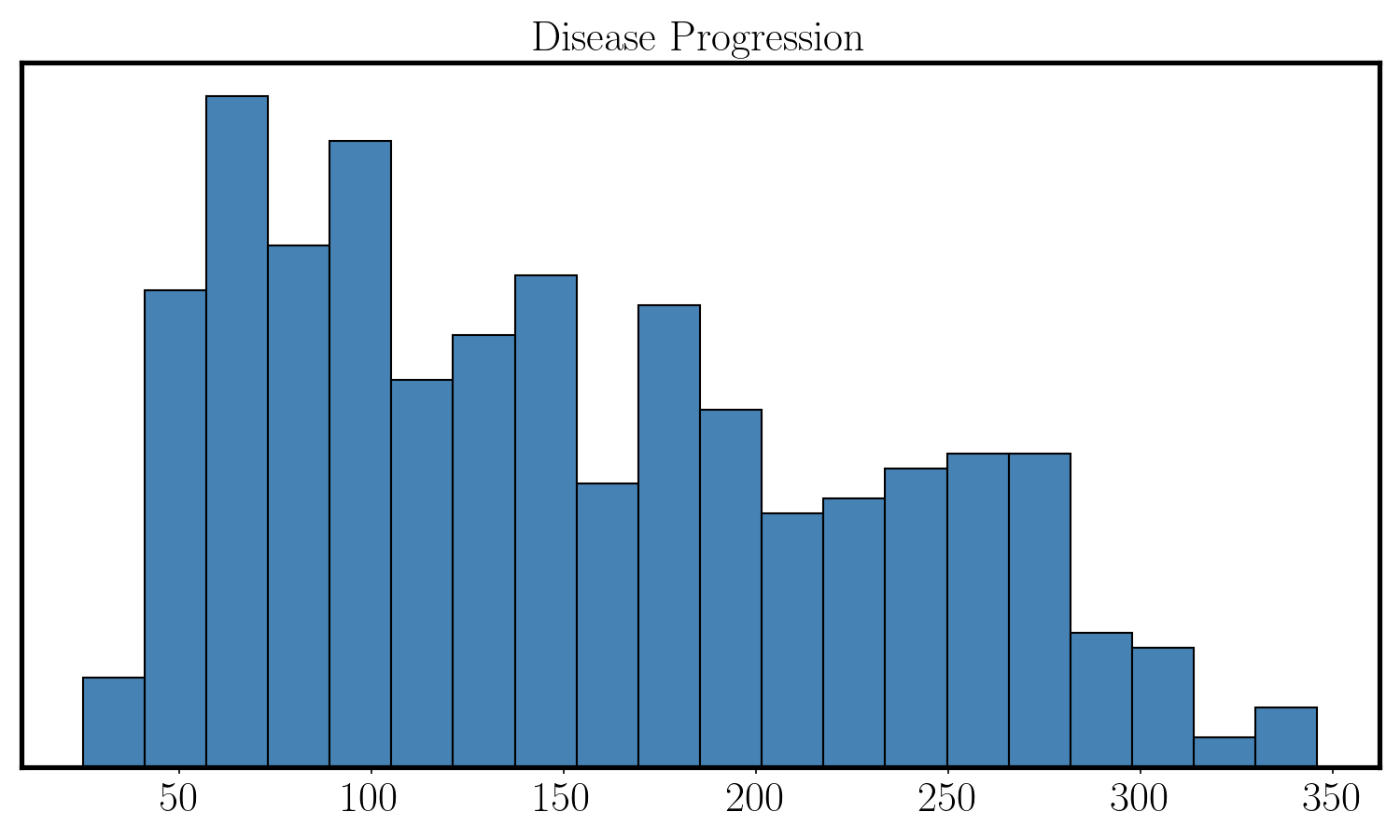}}
  \hfill
  \includegraphics[width=0.48\textwidth]{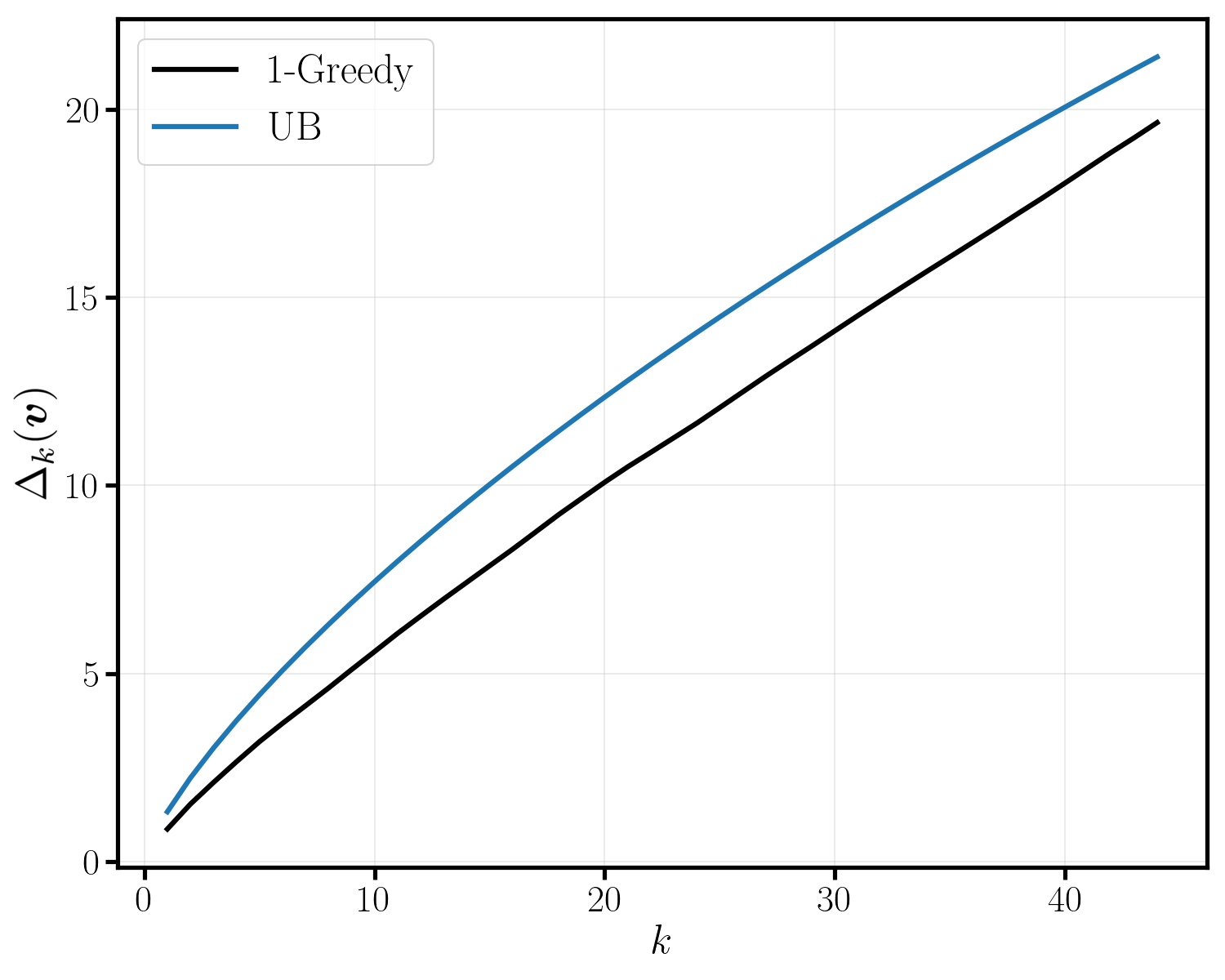}
  \caption{Diabetes dataset \citep{efron2004least}. Left: histogram of the response variable $y$, measuring disease progression. Right: empirical \texttt{1-Greedy} lower bound on $\Delta_k(\bv)$ (black), compared with the curve $C(k/n)\log(en/k)$ (blue), with $C=60$.}
  \label{fig:diabetes_y_greedy}
\end{figure}

\newpage
\section{Proof of Theorems \ref{lem1:conc_ineq_for_thrm1}--\ref{thrm:linmodel_nonasymp}}
\label{AppendixA}



Throughout,  $C,c > 0$ denote absolute constants whose values may vary from line to line. 

\subsection{Auxiliary Lemmas}

Our proofs involve the whitened data matrix $Z \coloneqq  X \Sigma^{-1/2} \in \mathbb{R}^{n \times p}$. 
Under Model \ref{mod1}, the rows of this matrix, which we denote by $\bz_1,\ldots, \bz_n$, are i.i.d.\ sub-Gaussian vectors with norm $\|\Sigma^{-1/2} \bx\|_{\psi_2}$ and identity covariance.
For non-empty $\cS \subseteq [n]$, define $Z_\cS$ analogously to $X_\cS$. 
For future use, we state a bound on the singular values of $Z_\cS$: 

\begin{lemma}[Theorem 4.6.1 of \cite{Vershynin12}]\label{lem:maxeig} {Let $\cS \subseteq [n]$ be a deterministic, non-empty set.} Under condition (i) of Model \ref{mod1},  for any $t > 0$,
\begin{equation*} 
 \begin{aligned} 
   & \P \bigg(\sigma_1(Z_\cS) \geq   \sqrt{|\cS|} + C\|\Sigma^{-1/2} \bx\|_{\psi_2}^2(\sqrt{p} +  t) \bigg) \leq e^{- t^2} , \\ & 
       \P \bigg(\sigma_p(Z_\cS) \leq   \sqrt{|\cS|} - C\|\Sigma^{-1/2} \bx\|_{\psi_2}^2(\sqrt{p} +  t) \bigg) \leq e^{-t^2} .
       \end{aligned}    
 \end{equation*}
\end{lemma}

\noindent The following parallel result provides tight bounds for the multivariate Gaussian distribution:

\begin{lemma}[Theorem 6.1 of \cite{wainwright}]\label{Lem:maxeig2} {Let $\cS \subseteq [n]$ be a  deterministic, non-empty set.}  Under condition (i) of Model \ref{mod2}, for any $t > 0$,
    
\begin{equation*} 
 \begin{aligned} 
   & \P \bigg(\sigma_1(Z_\cS) \geq   \sqrt{|\cS|} + \sqrt{p} + t \bigg) \leq e^{-t^2/2} ,\\
&       \P \bigg(\sigma_p(Z_\cS) \leq \sqrt{|\cS|} - \sqrt{p} - t \bigg) \leq e^{-t^2/2} .
 \end{aligned}    
 \end{equation*}
\end{lemma}

We will frequently use Bernstein's inequality for sub-exponential random variables:

\begin{lemma}[Corollary 2.8.3 of \cite{Vershynin12}] \label{lem:bern}
    Let  $ \xi_1,\ldots,\xi_n$ denote independent, mean zero, sub-exponential random variables. Then, for any $t > 0$, 
    \[
    \P\bigg( \frac{1}{n} \sum_{i=1}^n \xi_i \geq K \cdot t \bigg) \leq e^{-c n \min(t,t^2)} , 
    \]
    where 
    $K \coloneqq \max_{1 \leq i \leq n} \|\xi_i\|_{\psi_1}$. 
\end{lemma}


Recall that the product of two sub-Gaussian random variables is sub-exponential. Hence, Lemma \ref{lem:bern} can be used to bound quantities such as $\|\by\|^2$ or $\|Z \bv\|^2 = \sum_{i=1}^n (\bv^\top \spa \bz_i)^2$, where $\bv \in \mathbb{S}^{p-1}$ is fixed. In the latter case, 
by Lemma 2.7.7 of \cite{Vershynin12}, 
\begin{align} \label{fact12345}
\max_{1 \leq i \leq n} \big\|(\bv^\top \spa \bz_i)^2 - 1\big\|_{\psi_1} \leq 1 + \max_{1 \leq i \leq n} \big\|(\bv^\top \spa \bz_i)^2 \big\|_{\psi_1} \leq 1 +  \big\| \Sigma^{-1/2} \bx\|_{\psi_2}^2 . 
\end{align}
In the proof of Theorem \ref{lem1:conc_ineq_for_thrm1}, we will use that $\E y^2\leq \|y\|_{\psi_2}^2$. Indeed, by (\ref{eq:subgdef}) and Jensen's inequality, 
\begin{align} \label{y2}
    \E y^2 \leq \|y\|_{\psi_2}^2 \log \Big(\E \exp \big(y^2/\|y\|_{\psi_2}^2 \big) \Big) \leq e \|y\|_{\psi_2}^2 .  
\end{align}

\begin{lemma}[Theorem 2.26 of \cite{wainwright}]\label{lem:lipschitz} Let $\bz \sim \mathcal{N}(0,I_p)$ and $f: \mathbb{R}^p \rightarrow \mathbb{R}$ be $L$-Lipschitz with respect to the Euclidian norm. Then, for any $t > 0$,
\begin{align*}
    &    \P \big( f(\bz) - \E f(\bz) \geq t\big) \leq e^{-t^2/(2L^2)} ,   
  &  \E \big[ e^{t(f(\bz) - \E f(\bz))}\big] \leq e^{ t^2 L^2 / 2} . 
\end{align*}
\end{lemma}
\noindent As an immediate consequence of Lemma \ref{lem:lipschitz} and Jensen's inequality, for $\bz \sim \mathcal{N}(0,I_p)$ and $t > 0$,
\begin{align} \label{Gaussian_norm}
    \P \big( \|\bz\| \geq \sqrt{p} + t \big) \leq e^{-t^2/2} . 
\end{align}

Lemmas \ref{lem:min_lipschitz} and \ref{lem:sv_lipschitz} are closely related to Example 5.3.4 of \cite{boucheron} and Example 2.32 of \cite{wainwright}:
\begin{lemma} \label{lem:min_lipschitz}
    Let $f_1, \ldots, f_n $ denote real-valued $L$-Lipschitz functions. Then, both $\min_{i \in [n]} f_i$ and $\max_{i \in [n]} f_i$ are $L$-Lipschitz. 
\end{lemma}

\begin{lemma}   \label{lem:sv_lipschitz} For $\cS \subseteq [n]$ and $i \in [p]$, the function $f : \mathbb{R}^{|\cS| \times p} \rightarrow \mathbb{R}$ defined by $f(A) \coloneqq \sigma_i(A_\cS)$ is $1$-Lipschitz with respect to the Frobenius norm. 

\end{lemma}


\subsection{Proof of Theorems \ref{lem1:conc_ineq_for_thrm1} and \ref{thrm:consistency}}


\begin{proof}[Proof of Theorem \ref{lem1:conc_ineq_for_thrm1}] Recall that our goal is to bound $\|\widehat \bbeta - \widehat \bbeta_\cS\|$ uniformly over all subsets $\cS \subseteq [n]$ of size at least $n-k$. To this end, we first derive a bound for a fixed subset $\cS$. By the definitions of $\widehat \bbeta$ and $\widehat \bbeta_\cS$, given in (\ref{eq:OLS}) and (\ref{eq:betas_def}), 
\begin{align}
\widehat \bbeta  - \widehat \bbeta_\cS
&= \big(X^\top \spa X\big)^{-1}X^\top \by -\big(X_\cS^\top X_\cS\big)^{-1}X_\cS^\top \by_\cS        \nonumber\\
&= \big(X^\top \spa X\big)^{-1} \big(X^\top \by  - X^\top \spa X\big(X_\cS^\top X_\cS\big)^{-1}X_\cS^\top \by_\cS\big)  . \label{kqwe812}
\end{align} 
Inserting the identities $X^\top \spa X = X_\cS^\top X_\cS + X_{\cS^c}^\top X_{\cS^c}$, $X^\top \by = X_\cS^\top \by_\cS + X_{\cS^c}^\top \by_{\cS^c}$, and $X = Z \Sigma^{1/2}$ into (\ref{kqwe812}) yields the decomposition
\begin{align}  \label{eq:mju7nhy6}
\widehat \bbeta  - \widehat \bbeta_\cS
&=  \big(X^\top \spa X\big)^{-1} \Big(X_{\cS^c}^\top \by_{\cS^c}  - X_{\cS^c}^\top X_{\cS^c} \big(X_\cS^\top X_\cS\big)^{-1}X_\cS^\top \by_\cS\Big)   \nonumber\\
&= \Sigma^{-1/2}\big(Z^\top \spa Z\big)^{-1} Z_{\cS^c}^\top\Big( \by_{\cS^c}  -  Z_{\cS^c} \big(Z_\cS^\top Z_\cS\big)^{-1}Z_\cS^\top \by_\cS\Big) .
\end{align}
By the triangle inequality and the identity $\|Z^\top \spa Z\| = \sigma_p^{2}(Z)$, 
\begin{align} \label{eq8123-1}
\|\widehat \bbeta - \widehat \bbeta_\cS\|
& \leq   \|\Sigma^{-1/2}\| \frac{ \sigma_1(Z_{\cS^c})}{\sigma_p^2(Z)} \Big( \|\by_{\cS^c}\|  +  \big\|Z_{\cS^c} \big(Z_\cS^\top Z_\cS\big)^{-1}Z_\cS^\top \by_\cS \big\| \Big)       \nonumber\\
& \leq   \|\Sigma^{-1/2}\| \frac{ \sigma_1(Z_{\cS^c})}{\sigma_p^2(Z)} \bigg( \|\by_{\cS^c}\|  +  \frac{\sigma_1(Z_\cS)}{\sigma_p^2(Z_\cS)}\|\bv_\cS \| \|\by_\cS\| \bigg) ,
\end{align}
where $\bv_\cS \in \mathbb{R}^{|\cS^c|}$ is given by
\begin{align*}
\bv_\cS \coloneqq   Z_{\cS^c} \frac{ \big(Z_\cS^\top Z_\cS\big)^{-1}Z_\cS^\top \by_\cS}{\big\|\big(Z_\cS^\top Z_\cS\big)^{-1}Z_\cS^\top \by_\cS\big\|} . 
\end{align*}
 We now bound each term on the right-hand side of (\ref{eq8123-1}). 
 By Lemma \ref{lem:maxeig}, for $t_1, t_2 > 0$,  
 \begin{equation} \label{eq8124-1}
 \begin{aligned}
   & \P \bigg(\sigma_1(Z_{\cS^c}) \geq   \sqrt{|\cS^c|} + C\omega^2(\sqrt{p} +  t_1) \bigg) \leq e^{-t_1^2} , \\ & 
       \P \bigg(\sigma_p(Z) \leq   \sqrt{n} - C\omega^2(\sqrt{p} + t_2) \bigg) \leq e^{-t_2^2} ,
 \end{aligned}    
 \end{equation}
 \vspace{-.26cm}
\begin{equation} \label{eq8124-7}
 \begin{aligned}  &  \hspace{-.37cm} \P \bigg(\sigma_1(Z_\cS) \geq   \sqrt{|\cS|} + C\omega^2(\sqrt{p} +  t_2) \bigg) \leq e^{-t_2^2} , \\ \hspace{.1cm}  & \hspace{-.37cm}
       \P \bigg(\sigma_p(Z_\cS) \leq   \sqrt{|\cS|} - C\omega^2(\sqrt{p} + t_2) \bigg) \leq e^{-t_2^2} .
 \end{aligned}    
 \end{equation}
By Lemma  \ref{lem:bern}, 
 \begin{equation} 
 \begin{aligned}
           \P \Big( \| \by_{\cS} \|^2 \geq \big( \E y^2 + \|y^2 - \E y^2\|_{\psi_1} \big)|\cS|  \Big) \leq  e^{- c |\cS| } . 
 \end{aligned} \label{bf123c1sdjkf}    
 \end{equation}
Since $|\cS| \geq n/2$ and  $\E y^2 + \|y^2 - \E y^2\|_{\psi_1} $ is bounded by a constant multiple of $\|y\|_{\psi_2}^2$ (see (\ref{y2}) and Lemma 2.7.6 in \cite{Vershynin12}), it follows that
\begin{align}\label{jop1}
 \P \Big( \| \by_{\cS} \|^2 \geq   \|y\|_{\psi_2}^2  n \Big) \leq  e^{- c n } .
\end{align}

\noindent Similarly, for $t_3 > 0$, 
\begin{equation*}
    \begin{aligned}
    & \P \bigg( \| \by_{\cS^c} \|^2 \geq  \|y\|_{\psi_2}^2 \Big(1+ \frac{k t_3}{|\cS^c|} \Big)  |\cS^c|  \bigg) \leq  e^{- c |\cS^c| - c k t_3 } .
\end{aligned}
\end{equation*}
Since $|\cS^c| \leq k$, this implies 
\begin{align} 
    \label{182j3g3}
        & \P \Big( \| \by_{\cS^c} \|^2 \geq  \|y\|_{\psi_2}^2 (1+t_3) k \Big) \leq  e^{- c k t_3} .
\end{align}

\noindent Next, we bound $\|\bv_\cS\|$. Since $Z_{\cS^c}$ is independent of  $Z_{\cS}$ and $\by_\cS$, the vector $\bv_\cS$, conditional on $Z_{\cS}$ and $\by_\cS$, is sub-Gaussian. In particular, the coordinates of $\bv_\cS$ are conditionally  i.i.d.\ with mean zero, variance one, and sub-Gaussian norm at most $\omega$ (see  (\ref{fact12345})). Thus, 
analogously to  (\ref{182j3g3}),
    \begin{align} & \P \Big( \| \bv_\cS \|^2 \geq ( 1 + \omega^2)(1+t_3) k \Big) \leq  e^{- c k t_3} . \label{182j3g}
\end{align}

Now, we bound the second term on the right-hand side of (\ref{eq8123-1}). Assume $C\omega^2(\sqrt{p} + t_2) < \sqrt{n-k}$, in which case (\ref{eq8124-7}) provides a non-trivial bound on $\sigma_p(Z_\cS)$.
Using (\ref{eq8124-7}), (\ref{jop1}),  (\ref{182j3g}), 
and $\sqrt{1+\omega^2} \leq 1 + \omega$, 
we have that with probability at least $1-2e^{- t_2^2}  - e^{-c k t_3} - e^{-c n}$,
\begin{align} \label{eqa8e}
 \frac{\sigma_1(Z_\cS)\|\bv_\cS\|\| \by_\cS\|}{  \sigma_p^2(Z_\cS) } & \leq  (1+\omega)\|y\|_{\psi_2} \cdot \frac{ \sqrt{n-k} + C \omega^2(\sqrt{p} + t_2)}{\big(\sqrt{n-k} - C \omega^2(\sqrt{p} + t_2) \big)^2} \cdot (1+\sqrt{t_3}) \sqrt{nk} . 
\end{align}
\noindent
Combining (\ref{eq8123-1}), (\ref{eq8124-1}), (\ref{182j3g3}), and the preceding bound yields that, with probability at least $1- e^{-t_1^2} - 3e^{-t_2^2} -2e^{-ck t_3 }  - e^{-cn} $,
\begin{align} \label{eq2312g}
      \|\widehat \bbeta - \widehat \bbeta_\cS\| 
   \leq  & \,  C  (1+\omega) \|\Sigma^{-1/2}\| \|y\|_{\psi_2}  \cdot 
   \frac{  \sqrt{k} + C \omega^2(\sqrt{p} + t_1)}{\big(\sqrt{n}-C \omega^2(\sqrt{p} + t_2) \big)^2} \nonumber \\ & \cdot \bigg( (1+\sqrt{t_3}) \sqrt{k} + \frac{ \sqrt{n-k} + C \omega^2(\sqrt{p} + t_2)}{\big(\sqrt{n-k} - C \omega^2(\sqrt{p} + t_2) \big)^2} \cdot (1+\sqrt{t_3}) \sqrt{nk}\bigg) . 
\end{align}

Applying a union bound over $\{\cS \subseteq [n]: |\cS| \geq n-k\}$ and using the inequality
\begin{align} \label{eqNcK}
\sum_{\ell=1}^k \binom{n}{\ell}  \leq k \Big(\frac{e n}{k}\Big)^k ,
\end{align}
which holds for $ k \leq n/2$, we 
conclude that  (\ref{eq2312g}) holds uniformly in $\cS$ 
with probability at least $1- e^{2 k \log(en/k)}\big( e^{-t_1^2} + 3e^{-t_2^2} + 2e^{-ck t_3}  + e^{-cn} \big)$: 
%
\begin{align} \label{eq2312gasdf}
    \max_{\substack{\cS \subseteq [n], \\ |\cS| \geq n-k}}
  \|\widehat \bbeta - \widehat \bbeta_\cS\| & 
   \leq \frac{C\eta}{n} \cdot 
   \frac{  1 + C \omega^2(\sqrt{p/n} + t_2/\sqrt{n})  }{ \big( 1 - C \omega^2(\sqrt{p/n} + t_2/\sqrt{n}) \big)^4}  \cdot \big(k + \sqrt{kp} + \sqrt{k} t_1  \big)(1+\sqrt{t_3}) . 
\end{align}

Note that by (\ref{thrm3.2cond}), there exists a  constant $\delta > 0$ such that 
\[
C \omega^2\bigg( \sqrt{\frac{p}{n}} + \sqrt{\frac{2k}{n} \log\Big( \frac{en}{k} \Big) } + \frac{\delta}{\omega^2}  \bigg) \leq \frac{1}{2} . 
\]
To conclude the proof, set 
\begin{align*} 
 & t_1 \coloneqq \sqrt{2 k \log\Big(\frac{en}{k}\Big) }(1+t) , &   t_2 \coloneqq \sqrt{2k \log\Big(\frac{en}{k}\Big)} + \frac{\delta\sqrt{n}}{\omega^2}   , & & t_3 \coloneqq \frac{2k}{c} \log\Big(\frac{en}{k}\Big) (1+t)^2 , 
\end{align*}
where $t > 0$ is a free parameter. With this choice of $t_2$, the second term on the right-hand side of (\ref{eq2312gasdf}) is bounded by 24. Then, with probability at least $1 - 3(en/k)^{-kt^2} - 4e^{-n\delta^2/\omega^4}$,
\begin{align} \label{line123asd}
    \max_{\cS \subseteq [n]: |\cS| \geq  n-k}   \|\widehat \bbeta - \widehat \bbeta_\cS\|  
   \leq C \eta  \cdot \frac{1}{n} \cdot \bigg( \sqrt{k \log \Big(\frac{en}{k}\Big) } + \sqrt{kp}  \bigg) \cdot \sqrt{k \log \Big(\frac{en}{k} \Big)} (1+t)^2 . 
\end{align} 
The theorem follows from (\ref{line123asd}). 
\end{proof}

\begin{proof}[Proof of Theorem \ref{thrm:consistency}] 
Let $\gamma \coloneqq p/n$ and  $\Omega$ be the inverse sample covariance matrix of $Z$:
\begin{align} \label{eq:defOmega}
\Omega \coloneqq \bigg(\frac{1}{n} Z^\top  \spa Z \bigg)^{-1} .  
\end{align}
Substituting $X = Z \Sigma^{1/2}$ into (\ref{eq:OLS}) yields
\[
\widehat \bbeta = \Sigma^{-1/2} \big(Z^\top  \spa Z\big)^{-1} \Sigma^{-1/2} X^\top \by  = \frac{1}{n}\Sigma^{-1/2} \Omega \Sigma^{-1/2} X^\top \by .
\]
Thus, by the triangle inequality, 
\begin{equation} 
\begin{aligned} \label{eq:betaS-beta2345}
  \|  \widehat \bbeta - \bbeta \|  &~ \leq   \bigg\|\frac{1}{n} \Sigma^{-1/2} \Omega \Sigma^{-1/2} X^\top \by  -   \frac{1}{n}\Sigma^{-1} X^\top \by \bigg\| + \bigg\| \frac{1}{n}\Sigma^{-1} X^\top \by  - \bbeta \bigg\|. 
\end{aligned}        
   \end{equation}
The first term on the right-hand side can be bounded as follows: 
\begin{align} 
    \bigg\| \frac{1}{n}\Sigma^{-1/2} \Omega \Sigma^{-1/2} X^\top \by  -\frac{1}{n}\Sigma^{-1} X^\top \by \bigg\| &~ =  \bigg\| \frac{1}{n}\Sigma^{-1/2} ( \Omega - I_p) \Sigma^{-1/2} X^\top \by   \bigg\| \nonumber
  \\   &~ \leq \sqrt{\kappa(\Sigma)}   \big\| \Omega - I_p \big\| \left\| \frac{1}{n}\Sigma^{-1} X^\top \by  \right\|     \nonumber
  \\ &~ \leq \sqrt{\kappa(\Sigma)}   \big\| \Omega - I_p \big\|
  \bigg( \bigg\| \frac{1}{n}\Sigma^{-1} X^\top \by  -\bbeta \bigg\|   + \| \bbeta\|\bigg) .
  \nonumber
\end{align}
This leads to the following bound on $\| \widehat \bbeta - \bbeta\|$:  
\begin{equation} 
\begin{aligned} \label{qbok}
  \|  \widehat \bbeta - \bbeta \|  &~ \leq   \big(1 +  \sqrt{\kappa(\Sigma)}   \big\| \Omega - I_p \big\| \big) \left\| \frac{1}{n}\Sigma^{-1} X^\top \by - \bbeta  \right\| + \sqrt{\kappa(\Sigma)}   \big\| \Omega - I_p \big\| \|\bbeta\| . 
\end{aligned}        
   \end{equation}

 By Lemma \ref{lem:maxeig} for $t_1, t_2 > 0$, such that  $C \omega^2(\sqrt{\gamma} + t_2) < 1$, 
\begin{align}\label{eq:eigbound_2sided}
         \P \bigg( \|\Omega - I_p\| \geq \frac{\big(1+C\omega^2(\sqrt{\gamma} + t_1)\big)^2 - 1}{\big(1-C\omega^2(\sqrt{\gamma} + t_2)\big)^2} \bigg) \leq e^{-nt_1^2} + e^{-nt_2^2} .
\end{align}
Note that  if $\|A - I_p\| \leq \delta < 1$, then $A$ is invertible and $\|A^{-1} - I_p\| \leq \delta/(1-\delta)$. Using this fact together with  (4.22) in \cite{Vershynin12}, one obtains
%
\begin{align}\label{eq:eigbound_2sided2}
         \P \bigg( \|\Omega - I_p\| \geq \frac{C\omega^2(\sqrt{\gamma} + t_2)}{1-C\omega^2(\sqrt{\gamma} + t_2)} \bigg) \leq 2e^{-nt_2^2} .
\end{align}

To bound the second term on the right-hand side of (\ref{eq:betaS-beta2345}), notice that $X^\top \by = \sum_{i =1}^n y_i \bx_i$ is a sum of i.i.d.\ sub-exponential vectors. 
For any fixed $\bv \in \mathbb{S}^{p-1}$,  applying the two-sided analog of Lemma \ref{lem:bern} yields, for $t_3 > 0$,
    \begin{equation} \label{eq:vZe}
    \begin{aligned}
        & \P \bigg( \bigg|\frac{1}{n} \bv^\top \Sigma^{-1} X^\top \by - \bv^\top \spa \bbeta \bigg| \geq \big\|\Sigma^{-1}(y \bx - \E[y \bx]) \big\|_{\psi_1} \cdot t_3  \bigg) \leq 2e^{-c n t_3^2} .
    \end{aligned} \end{equation}
   By Corollary 4.2.13 of \cite{Vershynin12},  there exists a $1/2$-net  $\mathcal{N}$ covering $\mathbb{S}^{p-1}$ with cardinality  at most
    \[|\mathcal{N}| \leq 5^p  . \]  Then, using a union bound over $\mathcal{N}$ and the inequality
\begin{align*}  
    \bigg\|\frac{1}{n}\Sigma^{-1} X^\top \by  - \bbeta\bigg\| \leq 
    2 \max_{\bv \in \mathcal{N}} \Big( \frac1n\bv^\top \Sigma^{-1} X^\top \by  - \bv^\top \spa \bbeta \Big) \end{align*}
(Exercise 4.4.2 in \cite{Vershynin12}), we have 
\begin{align} \label{eq:deltanet2}
    \P\bigg(\bigg\|\frac{1}{n}\Sigma^{-1} X^\top \by  - \bbeta\bigg\| \geq 2 \big\|\Sigma^{-1}(y \bx - \E[y \bx]) \big\|_{\psi_1} \cdot t_3 \bigg) \leq 2e^{\log(5) p - c n t_3^2} . 
\end{align}

Thus, by  (\ref{qbok})--(\ref{eq:eigbound_2sided2}) and (\ref{eq:deltanet2}),  with probability at least $1 - e^{-nt_1^2} - 3e^{-nt_2^2} - 2 e^{\log(5) p  - cn t_3^2}$,
    \begin{align} 
     \|  \widehat \bbeta - \bbeta\| \leq &~  2 \bigg(1+\frac{C\sqrt{\kappa(\Sigma)} \omega^2(\sqrt{\gamma} + t_2)}{1-C\omega^2(\sqrt{\gamma} + t_2)}  \bigg) \big\|\Sigma^{-1}(y \bx - \E[y \bx]) \big\|_{\psi_1} \cdot t_3  \nonumber   \\
     &~ + \sqrt{\kappa(\Sigma)} \bigg(\frac{\big(1+C\omega^2(\sqrt{\gamma} + t_1)\big)^2 - 1}{\big(1-C\omega^2(\sqrt{\gamma} + t_2)\big)^2} \bigg) \|\bbeta\| .  \label{19287hctgfiy}    
\end{align}  
Let $\delta > 0$ be a constant such that  $C(\omega^2 \sqrt{\gamma} + \delta)  \leq 1/2$. To conclude the proof,  set 
\begin{align*}
& t_1 \coloneqq \sqrt{\gamma} \hspace{.15em} t ,  & t_2 \coloneqq  \frac{\delta}{\omega^2}, && t_3 \coloneqq \sqrt{\frac{\gamma}{c} \log(5)} \hspace{.15em} (1+ t) ,
\end{align*}
where $t >  0$ is a free parameter. Then, with probability at least $1 - 3e^{-pt^2} - 3e^{-n\delta^2/\omega^4}$,
\begin{align}  
     \|  \widehat \bbeta - \bbeta\| \leq  &~ C \big(1+\sqrt{\kappa(\Sigma)} \big) \big\|\Sigma^{-1}(y \bx - \E[y \bx]) \big\|_{\psi_1} \sqrt{\gamma}(1+t) \nonumber    \\
     &~ + C \sqrt{\kappa(\Sigma)}  \omega^2(1+\omega^2) \|\bbeta\| \big( \sqrt{\gamma}(1+t) + \gamma(1+t)^2\big) .   \label{zq}
\end{align}  
The theorem follows from (\ref{zq}).
\end{proof}

\comment{\color{red}
Under Mod 2, \[\widehat \bbeta - \bbeta = (X^\top X)^{-1} X^\top \beps = (n^{-1} X^\top X)^{-1} \frac{1}{n} X^\top \beps . \]
Consider $\Sigma = I_p$:
\[
\P\Big(\frac{1}{n} \|\beps\|^2 - 1 \geq t_1\Big) \leq e^{-c n t_1^2},
\]
Let $\widetilde \beps = \beps / \|\beps\|$.
\[
\P\Big( \frac{1}{p}\|Z^\top \widetilde \beps\|^2 - 1 \geq t_2 \Big) \leq e^{- c p t_2^2} . 
\]
Thus,
\[
\P\Big( \|Z^\top \beps\| \geq \sqrt{(1+t_1)(1+t_2)np} \Big) \leq e^{-cnt_1^2} + e^{-cpt_2^2}.
\]
\[
\P\Big( \frac{1}{n} \|Z^\top \beps\| \geq \frac{\sqrt{(1+t_1)(1+t_2)p} }{\sqrt{n}} \Big) \leq e^{-cnt_1^2} + e^{-cpt_2^2}.
\]
How do we reconcile the $e^{-cpt_2^2}$ with (26)? To conclude from (26)--(27), we make the change of variables $t = \widetilde t + \sqrt{p\log(1+2/\delta)/cn}$. Then the probability in (26) and (27) is just $e^{-cnt^2}$. 
}

\subsection{Proof of Theorem \ref{thrm:linmodel_nonasymp}} \label{seca3}

Throughout this section,  we assume $(\bx_1, y_1), \ldots, (\bx_n, y_n) $ are i.i.d.\ samples from Model \ref{mod2}, $\eps \sim \mathcal{N}(0, \sigma_\eps^2)$, $k \leq n/2$, and $p \leq n-k$. The proof of Theorem \ref{thrm:linmodel_nonasymp} is based on the following lemmas: 

\begin{lemma} \label{lem:Zeps}  For  $t, \delta  > 0$, with probability at least $1 - (en/k)^{-kt^2}  - e^{-\delta^2/2} - e^{-n/2}$,
\begin{align} \label{eq:lema7_1}
\max_{{\cS \subseteq [n], |\cS| \geq n-k}}    \|Z_\cS^\top \beps_\cS\|    & \leq    \sigma_\eps \bigg(3 k  \log \Big(\frac{en}{k}\Big) (1+t) + 2\sqrt{n}(\sqrt{p} +  \delta) \bigg)    ,
\end{align}
where $\beps_\cS$ is defined analogously to $\by_\cS$.
Furthermore, with probability at least $1  - (en/k)^{-kt^2} - e^{-\delta^2/2}$,
\begin{align} \label{eq:lema7_2}
\max_{{\cS \subset [n], |\cS| \geq n-k}}    \|Z_{\cS^c}^\top \beps_{\cS^c}\|    & \leq    \sigma_\eps \bigg(3 k  \log \Big(\frac{en}{k}\Big)  + (\sqrt{p} +  \delta) \sqrt{ 3 k  \log \Big(\frac{en}{k}\Big)} \,\bigg) (1+t).
\end{align}
\end{lemma}

\begin{lemma}
\label{lem:lb_sigma_p2} 
For $t > 0$, the following bounds hold individually with probability at least $1 - e^{-t^2/2}$:
       \begin{align} \label{lemmaA8.1}
     \min_{\cS \subseteq [n], |\cS| \geq n-k} \sigma_p(Z_\cS) & \geq  \sqrt{n} - \sqrt{p} -  \sqrt{3k \log \Big(\frac{en}{k}\Big)}-t , \\
     \label{lemmaA8.2}
     \max_{\cS \subset [n], |\cS| \geq n-k} \sigma_1(Z_{\cS^c}) & \leq \sqrt{3 k \log \Big(\frac{en}{k}\Big)} + \sqrt{p} + t  . 
    \end{align} 
\end{lemma}

\begin{proof}[Proof of Theorem \ref{thrm:linmodel_nonasymp}]
   Analogous to (\ref{eq:mju7nhy6}), $\widehat \bbeta - \widehat \bbeta_\cS$ is decomposable as
    \begin{align}  \label{g}
\widehat \bbeta  - \widehat \bbeta_\cS
&= \Sigma^{-1/2}\big(Z^\top \spa Z\big)^{-1} Z_{\cS^c}^\top\Big( \beps_{\cS^c}  -  Z_{\cS^c} \big(Z_\cS^\top Z_\cS\big)^{-1}Z_\cS^\top \beps_\cS\Big) .
\end{align}
Thus,
\begin{equation}
     \begin{aligned}  
\max_{\cS \subseteq [n], |\cS| \geq n-k} \| \widehat \bbeta  - \widehat \bbeta_\cS \| & \leq  \frac{\|\Sigma^{-1/2}\|}{\sigma_p^2(Z)} \cdot  \max_{\cS \subseteq [n], |\cS| \geq n-k}  \| Z_{\cS^c}^\top \beps_{\cS^c}\|   \\
    & +  \frac{\|\Sigma^{-1/2}\|}{\sigma_p^2(Z)} \cdot  \max_{\cS \subseteq [n], |\cS| \geq n-k}  \big \| Z_{\cS^c}^\top Z_{\cS^c} \big(Z_\cS^\top Z_\cS\big)^{-1}Z_\cS^\top \beps_\cS \big\|   .
\label{qqwuefsd}\end{aligned}
\end{equation}
Since (\ref{g}) is linear in $\beps$, we may assume without loss of generality that $\sigma_\eps = 1$.

To bound $\| Z_{\cS^c}^\top \beps_{\cS^c}\|$ uniformly, we use (\ref{eq:lema7_2}): with probability at least $1 - 2(en/k)^{-kt^2}$,
 \begin{align} \max_{\cS \subseteq [n], |\cS| \geq n-k}   \| Z_{\cS^c}^\top \beps_{\cS^c}\| 
\leq \bigg( 6 k\log \Big(\frac{en}{k}\Big) +  \sqrt{3 kp \log \Big(\frac{en}{k}\Big)} \hspace{.1em} \bigg)(1+t)^2 .  \label{eq:a26} 
\end{align}
Next, we bound $\big \| Z_{\cS^c}^\top Z_{\cS^c} \big(Z_\cS^\top Z_\cS\big)^{-1}Z_\cS^\top \beps_\cS \big\|$ uniformly.  For fixed $\cS$, since $Z_{\cS^c}$ is independent of $Z_\cS$ and $\beps_\cS$, (\ref{Gaussian_norm}) implies
\begin{align} \label{v}
\P\bigg( \big\| Z_{\cS^c} \big(Z_\cS^\top Z_\cS\big)^{-1}Z_\cS^\top \beps_\cS \big\| \geq  \big\| \big(Z_\cS^\top Z_\cS\big)^{-1}Z_\cS^\top \beps_\cS \big\| \big(\sqrt{k} +t\big) \bigg) \leq e^{-t^2/2} . 
\end{align}
Using a union bound and (\ref{eqNcK}), it follows from (\ref{v}) that with probability at least $1 - (en/k)^{- kt^2}$,
\begin{align} \label{x3rlop}
     &~ \max_{\substack{\cS \subseteq [n], \\ |\cS| \geq n-k}} \big\| Z_{\cS^c} \big(Z_\cS^\top Z_\cS\big)^{-1}Z_\cS^\top \beps_\cS \big\| \leq   3 \sqrt{k \log \Big( \frac{en}{k}\Big)} (1+t)  \cdot  \max_{\substack{\cS \subseteq [n], \\ |\cS| \geq n-k}}   \big\| \big(Z_\cS^\top Z_\cS\big)^{-1}Z_\cS^\top \beps_\cS \big\| . \end{align}
By (\ref{eq:lema7_1}), with probability at least $1 - 2(en/k)^{- kt^2} - e^{-n/2}$,
\begin{align} \label{x4rlop}
  \max_{{\cS \subseteq [n], |\cS| \geq n-k}}    \|Z_\cS^\top \beps_\cS\|    & \leq 2\sqrt{np} +   
  \bigg( 3 k  \log \Big(\frac{en}{k}\Big) +  \sqrt{8 n k  \log \Big(\frac{en}{k}\Big)} \bigg)(1+t)      \nonumber \\
  & \leq 2\sqrt{np} + 6\sqrt{ n k  \log \Big(\frac{en}{k}\Big)} (1+t).
\end{align}
Here, the second inequality follows from the bound $\log(x) \leq x/e$ for $x \geq 0$, which implies
\[
 k  \log \Big(\frac{en}{k}\Big) \leq \sqrt{nk \log \Big(\frac{en}{k}\Big)} . 
\]

Combining (\ref{lemmaA8.1}), (\ref{x3rlop}), and (\ref{x4rlop}),  we obtain that, with probability at least $1 - 3(en/k)^{- kt^2} - e^{-n\delta^2/2} - e^{-n/2}$,
\vspace{-.1cm}
\begin{align}  \nonumber 
    \max_{\substack{\cS \subseteq [n], \\ |\cS| \geq n-k}} \big\| Z_{\cS^c} \big(Z_\cS^\top Z_\cS\big)^{-1}Z_\cS^\top \beps_\cS \big\| \leq  \frac{6}{(1-\rho)^2}\bigg( &  \sqrt{\frac{kp}{n}\log \Big( \frac{en}{k}\Big)}(1+t)   +  \frac{3k}{\sqrt{n}}  \log \Big( \frac{en}{k}\Big) (1+t)^2 \bigg)   ,
\end{align}
where $\rho$ is defined in (\ref{rho_def}). The result follows from this bound, Lemma \ref{Lem:maxeig2},  
  (\ref{lemmaA8.2}), and (\ref{eq:a26}).
\comment{By  (\ref{lemmaA8.2}), with probability at least $1 - (en/k)^{- kt^2}$,
\begin{align*} 
    \max_{\cS \subseteq [n], |\cS| \geq n-k} \sigma_1(Z_{\cS^c}) \leq\sqrt{3 k \log \Big( \frac{en}{k}\Big)} (1+t) + \sqrt{p} .
\end{align*}
Together, these bounds imply that with probability at least $1 - 4(en/k)^{- kt^2} - e^{-n\delta^2/2} - e^{-n/2}$,
\comment{\begin{align}  
     &~ \max_{\cS \subseteq [n], |\cS| \geq n-k}\big\| Z_{\cS^c}^\top Z_{\cS^c} \big(Z_\cS^\top Z_\cS\big)^{-1}Z_\cS^\top \beps_\cS \big\| \label{final2}  \\  \leq &~    \frac{1}{(1-\rho)^2}\bigg(  8 p \sqrt{\frac{k}{n}\log \Big( \frac{en}{k}\Big)}(1+t)  +   \frac{38k \sqrt{p}}{\sqrt{n}}  \log \Big( \frac{en}{k}\Big) (1+t)^2 +  
     \frac{42 k^{3/2}}{\sqrt{n}}\log^{3/2}\hspace{-.15em}\Big( \frac{en}{k}\Big)(1+t)^3 \bigg)   . \notag
\end{align}}
\begin{align}  
     &~ \max_{\cS \subseteq [n], |\cS| \geq n-k}\big\| Z_{\cS^c}^\top Z_{\cS^c} \big(Z_\cS^\top Z_\cS\big)^{-1}Z_\cS^\top \beps_\cS \big\| \label{final2}  \\  \leq &~    \frac{1}{(1-\rho)^2}\bigg(  6 p \sqrt{\frac{k}{n}\log \Big( \frac{en}{k}\Big)}  +   \frac{30k \sqrt{p}}{\sqrt{n}}  \log \Big( \frac{en}{k}\Big) +  
     \frac{32 k^{3/2}}{\sqrt{n}}\log^{3/2}\hspace{-.15em}\Big( \frac{en}{k}\Big)(1+t) \bigg)(1+t)^2   . \notag
\end{align}

\noindent The theorem  follows from Lemma \ref{Lem:maxeig2}, (\ref{final2}), and 
(\ref{eq:a26}).  }
\end{proof}


Lemma \ref{lem:Zeps} follows from Lemmas  \ref{lem:Zeps2} and \ref{lem23j9874dx40} and (\ref{Gaussian_norm}).
\begin{lemma} \label{lem:Zeps2}
  For $t \geq 0$, the following bounds hold individually with probability at least $1 - e^{-t^2/2}$:
\begin{align}
\max_{{\cS \subseteq [n], |\cS| \geq n-k}} \|Z_\cS^\top \beps_\cS\|  &\leq   \sqrt{3 k \log\Big(\frac{en}{k}\Big)} \cdot  \max_{\cS \subset [n], |\cS| = k} \|\beps_\cS\| + (\sqrt{p} + t) \|\beps\|   , \label{jadhyads1} \\
\max_{{\cS \subset [n], |\cS| \geq n-k}} \|Z_{\cS^c}^\top \beps_{\cS^c}\| & \leq \bigg(  \sqrt{3 k \log\Big(\frac{en}{k}\Big)} + \sqrt{p} + t \bigg)  \max_{\cS \subset [n], |\cS| = k} \|\beps_\cS\|  . \label{jadhyads2}
\end{align}
\end{lemma}

\begin{lemma} \label{lem23j9874dx40}  For $t \geq 0$, with probability at least $1 - (e n/k)^{-kt^2}$, 
\begin{align} \label{123cas}
    \max_{\cS \subset [n], |\cS| = k}    \|  \beps_\cS \| \leq  \sigma_\eps \sqrt{3k \log \Big(\frac{en}{k}\Big)} \cdot (1+t) .
    \end{align}
\end{lemma}

 \begin{proof}[Proof of Lemma \ref{lem:Zeps2}] We prove (\ref{jadhyads1}); the proof of (\ref{jadhyads2}) is similar and omitted. 
 
 Since $Z \mapsto \|Z_\cS^\top \beps_\cS\|$ is $\|\beps\|$-Lipschitz and $Z$ and $\beps$ are independent, Lemmas \ref{lem:lipschitz} and \ref{lem:min_lipschitz} yield that for any $t > 0$,
\begin{align} \label{eq:pgb1}
    \P \bigg(\max_{{\cS \subseteq [n], \cS \geq n-k}} \|Z_\cS^\top \beps_\cS\| \geq \E\Big[ \max_{{\cS \subseteq [n], \cS \geq n-k}}   \|Z_\cS^\top \beps_\cS\| \, \big|\, \beps \Big] + t \|\beps\| \,\Big|\, \beps \bigg) \leq e^{-t^2/2} .
\end{align} 

We bound the expectation in (\ref{eq:pgb1}) by modifying the proof of Theorem 6.1 in \cite{wainwright}.
Let $\bxi \sim \mathcal{N}(0,I_n)$ and $\bzeta \sim \mathcal{N}(0,I_p)$ be independent
and 
\[ Y_{\bv,\cS} \coloneqq  \beps_\cS^\top \bxi_\cS + \|\beps\| \bv^\top \spa \bzeta,
\]
where $\bxi_\cS$ is defined analgously to $\by_\cS$. 
Then, conditional on $\beps$, $\{Y_{\bv,\cS}\}$ is a centered Gaussian process over the (separable) index set 
$\{\bv \in \mathbb{S}^{p-1},|\cS| \geq n-k\}$. Importantly, its increments bound those of  $ \beps_\cS^\top Z_\cS \bv$: following page 164 of \cite{wainwright},
    \begin{align} \nonumber
            \E \Big[ \big(\beps_\cS^\top Z_\cS \bv - \beps_{\widetilde \cS}^\top Z_{\widetilde \cS} \, \widetilde \bv  \big)^2  \,\Big|\, \beps \Big] &= \E \Big[ \big( \beps_\cS^\top  I_{n,\cS}  Z \bv - \beps_{\widetilde \cS}^\top I_{n,\widetilde \cS} \, Z \widetilde \bv  \big)^2 \,\Big|\, \beps \Big] \\ &=  \big\| \big( I_{n,\cS}^\top \beps_\cS\big) \bv^\top - \big(I_{n,\widetilde \cS}^\top \, \beps_{\widetilde \cS}\big) \, \widetilde \bv^\top \big\|_F^2 \nonumber \\
    & \leq \big\|I_{n,\cS}^\top \beps_\cS -   I_{n,\widetilde \cS}^\top\,  \beps_{\widetilde \cS} \big\|^2 + \|\beps\|^2 \|\bv - \widetilde \bv\|^2 \nonumber \\ & =   \E \Big[ \big(Y_{\bv,\cS} - Y_{\widetilde \bv, \widetilde \cS} \big)^2 \, \Big| \, \beps \Big]. \label{dfyughunw}
    \end{align} 
Here,  $I_{n,\cS} \in \mathbb{R}^{|\cS| \times n}$ is a submatrix of the identity matrix $I_n$ (defined analogously to $Z_\cS$), and  the final equality holds by the independence of $\bxi$ and $\bzeta$ and the identity
\[\beps_\cS^\top \bxi_\cS - \beps_{\widetilde \cS}^\top \bxi_{\widetilde \cS} = (\beps_{\cS}^\top I_{n,\cS} -\beps_{\widetilde\cS}^\top I_{n,\widetilde \cS})\bxi .\]
%
%
Therefore, by the Sudakov--Fernique inequality (Theorem 2.2.3 of \cite{adler2007random}) and the bound $\E \|\bzeta\| \leq \sqrt{\E \|\bzeta\|^2} = \sqrt{p}$, we have 
    \begin{align}\label{`8923u1ejsd}
       \E\Big[ \max_{{\cS \subseteq [n], \cS \geq n-k}}   \|Z_\cS^\top \beps_\cS\| \, \Big|\, \beps \Big]  & = 
       \E\Big[ \max_{{\cS \subseteq [n], \cS \geq n-k}} \sup_{\bv \in \mathbb{S}^{p-1}}   \beps_\cS^\top Z_\cS \bv \, \Big|\, \beps \Big]  \nonumber \\
       & \leq  \E\Big[ \max_{{\cS \subseteq [n], \cS \geq n-k}} \sup_{\bv \in \mathbb{S}^{p-1}}   Y_{\bv,\cS} \, \Big|\, \beps \Big]  \nonumber \\
       & \leq \E \Big[ \max_{\cS \subseteq [n], |\cS| \geq n-k}      \beps_\cS^\top \bxi_\cS  \, \Big| \, \beps \Big] + \sqrt{p} \|\beps\| .     
    \end{align}

Since $\E[\beps^\top \bxi | \beps ]  = 0 $, the expectation on the right-hand side of (\ref{`8923u1ejsd}) may be written as
\begin{align*}
     \E \Big[ \max_{\cS \subseteq [n], |\cS| \geq n-k}      \beps_\cS^\top \bxi_\cS  \, \Big| \, \beps \Big] & =  \E \Big[ \max_{\cS \subseteq [n], |\cS| \geq n-k}      -\beps_{\cS^c}^\top \bxi_{\cS^c}  \, \Big| \, \beps \Big] \\
     &  =  \E \Big[ \max_{\cS \subset [n], |\cS| \leq k}    \beps_{\cS}^\top \bxi_{\cS}  \, \Big| \, \beps \Big] . 
\end{align*}
%
For any $\eta > 0$, using Jensen's inequality, 
    \begin{align}  \label{eq:polikuj}
       \exp \bigg( \eta \, \E \Big[ \max_{\cS \subset [n], |\cS| \leq k}      \beps_\cS^\top \bxi_\cS  \, \Big| \, \beps \Big] \bigg) & \leq  \E \Big[ \max_{\cS \subset [n], |\cS| \leq k}      \exp\big(\eta\beps_\cS^\top \bxi_\cS\big)  \, \Big| \, \beps \Big]  \nonumber\\
        & \leq \sum_{\cS \subset [n], |\cS| \leq k}    \E \Big[     \exp\big(\eta\beps_\cS^\top \bxi_\cS\big)  \, \Big| \, \beps \Big]   \nonumber\\        & \leq \sum_{\cS \subset [n], |\cS| \leq k} \exp\bigg( \frac{\eta^2 \|\beps_\cS\|^2}{2} \bigg)  \nonumber 
        \\
        & \leq \big|\{\cS \subset [n]: |\cS| \leq k \} \big| \cdot  \exp\bigg( \frac{\eta^2}{2} \max_{\cS \subset [n], |\cS| = k} \|\beps_\cS\|^2 \bigg) . 
    \end{align}
Thus, using (\ref{eqNcK}),
\begin{align*}
            \E \Big[ \max_{\cS \subset [n], |\cS| \leq k}      \beps_\cS^\top \bxi_\cS  \, \Big| \, \beps \Big]  &\leq \frac{\log k}{\eta} + \frac{k}{\eta}\log\Big(\frac{en}{k}\Big) + \frac{\eta }{2} \cdot \max_{\cS \subset [n], |\cS| = k} \|\beps_\cS\|^2  .
\end{align*}
Choosing $\eta$ such that 
\[
\eta \cdot \max_{\cS \subset [n], |\cS| = k} \|\beps_\cS\| = \sqrt{2k\log\Big(\frac{en}{k}\Big)} ,\]
it is easy to show using $k \leq n/2$ that
\begin{align} \label{eq:pgb3}
            \E \Big[ \max_{\cS \subset [n], |\cS| \leq k}      \beps_\cS^\top \bxi_\cS  \, \Big| \, \beps \Big]  
             & \leq   \sqrt{3 k \log\Big(\frac{en}{k}\Big)} \cdot  \max_{\cS \subset [n], |\cS| = k} \|\beps_\cS\| . 
\end{align}
The proof follows from (\ref{eq:pgb1}), (\ref{`8923u1ejsd}), and (\ref{eq:pgb3}). 
\end{proof}

\begin{proof}[Proof of Lemma \ref{lem23j9874dx40}]

The proof of (\ref{123cas}) is similar to (\ref{eq:pgb1})--(\ref{eq:pgb3}). Since $\beps \mapsto \|\beps_\cS\|$ is $1$-Lipschitz, Lemmas \ref{lem:lipschitz} and Lemma \ref{lem:min_lipschitz} yield
\begin{gather}
    \P \bigg(  \max_{\cS \subset [n], |\cS| = k}    \|  \beps_\cS \| \geq \E \Big[ \max_{\cS \subset [n], |\cS| = k}    \|  \beps_\cS \| \Big] +  \sigma_\eps t    \bigg) \leq e^{-t^2/2} , \label{vsuhhwvd8f} \\  
     \E \Big[ \exp\big(\eta (\|\beps_\cS\| -\E \|\beps_\cS\| ) \big)\Big] \leq \exp\bigg( \frac{\eta^2 \sigma_\eps^2 }{2}  \bigg) . \nonumber  
\end{gather}
Thus, for $\eta > 0$, 
\begin{gather*}
    \exp \Big(\eta \, \E \Big[ \max_{\cS \subset [n], |\cS| = k}    \|  \beps_\cS \|\Big] \Big) \leq \sum_{\cS \subset [n], |\cS| = k} \E \Big[ \exp \big( \eta \|\beps_\cS\|\big)\Big] \leq 
    \Big(\frac{en}{k}\Big)^k \exp \bigg( \frac{\eta^2 \sigma_\eps^2 }{2}   + \eta \, \E \|\beps_\cS\| \bigg) ,    \\
     \E \Big[ \max_{\cS \subset [n], |\cS| = k}    \|  \beps_\cS \|\Big]   \leq \frac{k}{\eta} \log\Big(\frac{en}{k}\Big) + \frac{\eta\sigma_\eps^2}{2} + \sigma_\eps \sqrt{k} .  
\end{gather*}
Optimizing the choice of $\eta$, 
we obtain
\begin{align}
    \E \Big[ \max_{\cS \subseteq [n], |\cS| = k}    \|  \beps_\cS \| \Big] 
    \leq \sigma_\eps \sqrt{3k \log \Big(\frac{en}{k}\Big)} .  \label{vsuhhwvd8f2}
\end{align}
The proof follows from (\ref{vsuhhwvd8f}) and (\ref{vsuhhwvd8f2}).

\end{proof}

\begin{proof}[Proof of Lemma \ref{lem:lb_sigma_p2}]  We first prove (\ref{lemmaA8.1}).    By Lemmas \ref{lem:lipschitz}--\ref{lem:sv_lipschitz},
    \begin{align} \label{eq:gaus_lip}
    \P \bigg(    \min_{\cS \subseteq [n], |\cS| \geq n-k} \sigma_p(Z_\cS) \leq \E \Big[ \min_{\cS \subseteq [n], |\cS| \geq n-k} \sigma_p(Z_\cS) \Big] - t \bigg) \leq e^{-t^2/2} .     
    \end{align} 
It therefore suffices to prove that
\begin{align} \label{eq:jkl45}
    \E \Big[ \min_{\cS \subseteq [n], |\cS| \geq n-k} \sigma_p(Z_\cS) \Big] \geq \sqrt{n} - \sqrt{p} - \sqrt{3k \log \Big(\frac{en}{k}\Big)}  . 
\end{align}

\noindent Since $p  \leq n-k$, 
    \begin{gather} 
        -\sigma_p(Z_\cS) =   \sup_{\bv \in \mathbb{S}^{p-1}} ( -\|Z_\cS v \|) = \sup_{\bv \in \mathbb{S}^{p-1}} \inf_{\bu \in \mathbb{S}^{|\cS|-1}} \bu^\top Z_\cS \bv , \nonumber \\
        \min_{\cS \subseteq [n], |\cS| \geq n-k} \sigma_p(Z_\cS)  = - \max_{\cS \subseteq [n], |\cS| \geq n-k}     \sup_{\bv \in \mathbb{S}^{p-1}} \inf_{\bu \in \mathbb{S}^{|\cS|-1}} \bu^\top Z_\cS \bv . \label{eq:maxmin}   
    \end{gather}
Let $\bxi \sim \mathcal{N}(0,I_n)$ and $\bzeta \sim \mathcal{N}(0,I_p)$ be independent
and define \[Y_{\bu,\bv,\cS} \coloneqq  \bu^\top \bxi_\cS + \bv^\top \spa \bzeta.\]
Then, $\{Y_{\bu,\bv,\cS}\}$ is a centered Gaussian process over the (separable) index set $\{\bu \in \mathbb{S}^{|\cS|-1}, \bv \in \mathbb{S}^{p-1}, |\cS| \geq n-k\}$, with increments bounding those of   $\bu^\top Z_\cS \bv$: ,
\begin{equation*}
    \begin{aligned}
            \E \big(\bu^\top Z_\cS \bv - \widetilde \bu^\top Z_{\widetilde \cS} \, \widetilde \bv\big)^2 &= \E \big( \bu^\top I_{n,\cS}  Z \bv - \widetilde \bu^\top  I_{n,\widetilde \cS} Z \widetilde \bv\big)^2 \\
             &=  \big\|  \big(I_{n,\cS}^\top \bu\big) \bv^\top - \big(  I_{n,\widetilde \cS}^\top \, \widetilde \bu\big) \widetilde \bv^\top \big\|_F^2 \\
    & \leq \big\|I_{n,\cS}^\top \bu -   I_{n,\widetilde \cS}^\top \, \widetilde \bu\big\|^2 + \|\bv - \widetilde \bv\|^2 \\ & =   \E \big(Y_{\bu,\bv,\cS} - Y_{\widetilde \bu, \widetilde \bv, \widetilde \cS}\big)^2 . 
    \end{aligned}
\end{equation*}
Therefore, Gordon's inequality ((6.65) in \cite{wainwright}) and (\ref{eq:maxmin}) yield 
    \begin{align} \label{eq:gordon}
       -\E \Big[ \min_{\cS \subseteq [n], |\cS| \geq n-k} \sigma_p(Z_\cS) \Big] & 
       \leq \E \Big[ \max_{\cS \subseteq [n], |\cS| \geq n-k}     \sup_{\bv \in \mathbb{S}^{p-1}} \inf_{\bu \in \mathbb{S}^{|\cS|-1}} Y_{\bu,\bv,\cS} \Big] \nonumber\\
       & \leq \E \Big[ \max_{\cS \subseteq [n], |\cS| \geq n-k}     \inf_{\bu \in \mathbb{S}^{|\cS|-1}}  \bu^\top \bxi_\cS \Big] + \E \| \bzeta \| \nonumber\\
       &  = - \E \Big[ \min_{\cS \subset [n], |\cS| = n-k}   \|\bxi_\cS\| \Big] + \E \| \bzeta \| \nonumber  
       \\& \leq - \E \|\bxi\| + \E \Big[ \max_{\cS \subset [n], |\cS| = k}   \|\bxi_\cS\| \Big]  +\E \|\bzeta\| ,
    \end{align}
    where the last inequality  follows from 
\[ \|\bxi\| \leq  
\max_{\cS \subset [n], |\cS| = k}   \|\bxi_\cS\| + \min_{\cS \subset [n], |\cS| =  n-k}   \|\bxi_\cS\| .
\]
Since $- \E \|\bxi\| + \E \|\bzeta\| \leq -\sqrt{n} + \sqrt{p}$ (see page 188 of \cite{wainwright}),  (\ref{eq:jkl45}) follows from (\ref{vsuhhwvd8f2}) and (\ref{eq:gordon}), completing the proof of (\ref{lemmaA8.1}).

Next, we prove (\ref{lemmaA8.2}).
     By Lemmas \ref{lem:lipschitz}--\ref{lem:sv_lipschitz},
    \begin{align} \label{eq:gaus_lipasdf}
    \P \bigg(    \max_{\cS \subset [n], |\cS| \geq n-k} \sigma_1(Z_{\cS^c}) \geq \E \Big[ \max_{\cS \subset [n], |\cS| \geq n- k} \sigma_1(Z_{\cS^c}) \Big] + t \bigg) \leq e^{-t^2/2} .     
    \end{align} 
By the Sudakov--Fernique inequality, 
\begin{align} \label{asd98yadsfg345c2t53}
    \E \Big[ \max_{\cS \subset [n], |\cS| \geq n-k} \sigma_1(Z_{\cS^c}) \Big] & \leq \E \Big[ \max_{\cS \subseteq [n], |\cS| \geq n-k}     \sup_{\bu \in \mathbb{S}^{|\cS^c|-1}} \sup_{\bv \in \mathbb{S}^{p-1}} Y_{\bu,\bv,\cS^c} \Big] \nonumber . \\
    & \leq \E \Big[ \max_{\cS \subset [n], |\cS| = k}   \|\bxi_\cS\| \Big] + \sqrt{p} . 
\end{align}
The result follows from (\ref{vsuhhwvd8f2}), (\ref{eq:gaus_lipasdf}), and (\ref{asd98yadsfg345c2t53}).
\end{proof}

\section{Proof of Theorem \ref{thm:asymp_small_p_large_k}}

Theorem \ref{thm:asymp_small_p_large_k} follows from Theorem \ref{thrm:small_p_large_k}. Throughout this appendix,  $C,c > 0$ denote absolute constants whose values may vary from line to line.  In the context of Model \ref{mod2}, define 
\begin{align} Q_{\alpha,t} \coloneqq (1-t)\E\big[ \varepsilon z \,\ind(\varepsilon z>q_{1-\alpha+t})\big] - t \| \varepsilon z \|_{\psi_1} , \label{defQ} \end{align}
where $z\sim \mathcal{N}(0,1)$ is  independent of $\varepsilon$ and $q_{1-\alpha+t}$ is the $(1-\alpha+t)$-quantile of $\varepsilon z$.

\begin{theorem}\label{thrm:small_p_large_k}
    Let $(\bx_1, y_1), \ldots, (\bx_n, y_n) $ be i.i.d.\ samples from Model \ref{mod2},
    $\gamma \coloneqq (p-1)/(n-k) < 1$,
    $\alpha \coloneqq k/n <1/2$,
     $\bv \in \mathbb{S}^{p-1}$, and $\eta\coloneqq \sqrt{2\E |\eps|^2 + \|\eps\|_{\psi_2}^2}$. Then, there exists a constant  $c_\eps >0$ depending on the distribution of $\eps$ such that for any $t, \delta \in (0,c_\eps)$, with probability at least
\begin{equation}\label{eq:prob_lower_bound}
1 - 13 e^{-c n t^2} - 2^{-c (n-k-p)\delta^2} - 3 e^{-n(1/2-\alpha)^2},
\end{equation}
the following inequality holds:
    \begin{align}
    \label{eq:small_p_large_k}
    &\Delta_k(\bv) \geq \big\|\Sigma^{-1/2}\bv\big\|\left(
        \frac{1-\gamma}{1-\gamma-t}\bigg(\frac{Q_{\alpha,t}}{1-\alpha+3t} -  3\hspace{.05em} \eta \hspace{.05em} \delta\bigg) - \frac{\eta \hspace{.05em} t}{(1-\sqrt{\gamma(1-\alpha)} -t)^2}
    \right).
    \end{align}
  \end{theorem}

\begin{remark}\label{rem:c_eps}
The restriction $t,\delta\in(0,c_\varepsilon)$ ensures the
leading term in \eqref{eq:small_p_large_k} is positive. Indeed, this term is positive for all sufficiently small $t, \delta > 0$ since $Q_{\alpha,t}$ is continuous in $t$ and  
$Q_{\alpha,0}>0$, the latter inequality following from the symmetry of the distribution of $\eps z$ and $\alpha < 1/2$.
\end{remark}

 Theorem \ref{thm:asymp_small_p_large_k} follows from Theorem \ref{thrm:small_p_large_k} by taking $t, \delta \rightarrow 0$, in which case (\ref{eq:prob_lower_bound}) converges to one and \eqref{eq:small_p_large_k} converges to (\ref{eq:asymp_small_p_large_k}). In particular, taking $t, \delta = O(n^{-1/2 + \iota})$ for $\iota \in (0, 1/2)$, we have
\begin{align*}
    \Delta_k(\bv) \geq ({1-\alpha})^{-1} \big\|\Sigma^{-1/2}\bv\big\|\E\big[ \varepsilon z \,\ind(\varepsilon z>q_{1-\alpha})\big]
    + O_\mathbb{P}\big(n^{-1/2 + \iota}\big).    
\end{align*}

\comment{\eqref{eq:small_p_large_k} for $n\gg 1$. 
In this case, we can take $\delta$ and $t$
to be $O(1/\sqrt{n})$, and by \eqref{eq:prob_lower_bound}, still have a relatively high probability that the bound 
\eqref{eq:small_p_large_k} holds. 
This implies that the second term in \eqref{eq:small_p_large_k} is $O(1/\sqrt{n})$
and is thus negligible with respect to the first term. 
Hence, using the definition of $Q_{\alpha,t}$, 
\begin{align}
\label{eq:Deltak_LB}
    \Delta_k(\bv) \geq \big\|\Sigma^{-1/2}\bv\big\| (1-\gamma)({1-\alpha})^{-1}\E\big[ \varepsilon z \,\ind(\varepsilon z>q_{1-\alpha})\big]
    + O(1/\sqrt{n}).    
\end{align}}

\subsection{Auxiliary Lemmas}
We will use the following lemmas,
whose proofs are deferred to Appendix~\ref{APBLEMM}. In the setting of
Theorem~\ref{thrm:small_p_large_k}, let
$Z \coloneqq X\Sigma^{-1/2} \in \mathbb{R}^{n\times p}$ denote the whitened
data matrix, and let $\beps \coloneqq (\eps_1,\ldots,\eps_n)^\top$.

\begin{lemma} 
    \label{lemma:ub_ZTZ_e1}
    For any $\bv \in \mathbb{S}^{p-1}$ and $t\in \big(0, \, 1-\sqrt{p/(n-1)}\,\big)$, with probability at least $1 - 2e^{-(n-1)t^2/2}$, 
    \begin{align*}
    \be_1^\top (Z^\top  \spa Z)^{-1}Z^\top \bv \geq - \frac{t}{ \sqrt{n-1} \big(1-\sqrt{p/(n-1)}-t\,\big)^2} . 
    \end{align*}
\end{lemma}

Let $G_{n,p}$ denote the Grassmannian consisting of all
$p$-dimensional subspaces of $\mathbb{R}^n$.
\begin{lemma}
\label{lemma:id_e1_inv(ZsZs)_Zs_eps}
Let $\cS$ contain the indices of the smallest $n-k$ values in $ \{\eps_i z_{i1}\}_{i\in [n]}$. Then, 
\begin{align} \label{qwertyuiop1}\be_1^\top \big(Z_\cS^\top Z_\cS\big)^{-1}Z_{\cS}^\top \beps_\cS \stackrel{\,d\,}{=} \be_1^\top \big(Z_\cS^\top Z_\cS\big)^{-1} \be_1\big( \be_1^\top Z_\cS^\top P \beps_\cS\big) ,  \end{align} 
where $P$ is the orthogonal projection onto a
subspace drawn uniformly from 
$G_{n-k,n-k-p+1}$, independently of $(\cS,  Z_\cS \be_1, \beps_\cS)$. Moreover, for any $t\in (0,1)$,
\begin{align} \label{qwertyuiop} 
\P\bigg( \Big| \Big( \|Z_{\cS} \be_1\|^2\be_1^\top \big(Z_\cS^\top Z_\cS\big)^{-1} \be_1\Big)^{-1} - 1 + \gamma \Big| \geq t \bigg) \leq 2 e^{-c(n-k-p+1)t^2} ,
\end{align}
where $\gamma \coloneqq (p-1)/(n-k)$.
\end{lemma}

The following result is a special case of the Johnson--Lindenstrauss lemma:
\begin{lemma}
    \label{lemma:JL_product}
Let $P$ be the orthogonal projection onto a subspace drawn
uniformly from $G_{n,p}$, and let $ \bv, \bw \in\mathbb{S}^{n-1}$.
Then, for any $\delta \in (0,1]$, with probability at least $1-2^{-cp \delta^2}$,
    \[
     \bv^\top P \bw
    \leq 
    \frac{p}{n}\big(\bv^\top\bw  + \delta \big).
    \]
\end{lemma}

\comment{Next, we state a simple claim about sub-exponential norms.
\begin{lemma}
\label{lemma:sub_exp_norm_ineq}
Let $u$ be sub-exponential random variable.
For any fixed $s\in \mathbb{R}$
    \begin{align}
        \label{eq:sub_exp_norm_ineq}
        \|u\ind(u>s)\|_{\psi_1} \leq \|u\|_{\psi_1}.
    \end{align}
\end{lemma}
    
\begin{proof}
For every $s\in \mathbb R$ and $t>0$
    \[
    \exp\bigg(\frac{|u|\ind(u>s)}{t}\bigg) \leq \exp\bigg(\frac{|u|}{t}\bigg).
    \]
Hence,
    \begin{align}
    \|u\ind(u>s)\|_{\psi_1}&=
    \inf\left\{ t > 0 : \mathbb{E}\left[\exp\bigg(\frac{|u|\ind(u>s)}{t}\bigg) \right]\leq e \right\} 
    \nonumber \\&
    \leq \inf\left\{ t > 0 : \mathbb{E}\left[\exp\bigg(\frac{|u|}{t}\bigg) \right]\leq e \right\} 
    \nonumber
    = \|u\|_{\psi_1}.
    \end{align}
    
\end{proof}}


The next lemma provides a bound on the sum of the
$n-k$ smallest values among $n$ i.i.d.\ sub-exponential random variables.
\begin{lemma} \label{lemma:lb_sum_w1}
    Let $\xi_1, \ldots, \xi_n$ be independent realizations of a sub-exponential  random variable $\xi$ with a mean zero,  symmetric distribution. Denote the order statistics by $\xi_{(1)} \leq \xi_{(2)} \leq \cdots \leq \xi_{(n)}$. 
    For any $\alpha \coloneqq k/n<1/2$ and $t \in (0, \alpha)$,
    \begin{align} \nonumber 
\P \Bigg(\frac{1}{n} \sum_{i = 1}^{n-k} \xi_{(i)} \geq  -(1-t)\E\big[ \xi \ind(\xi \geq q_{1-\alpha+t})\big]+ t \|\xi\|_{\psi_1} \Bigg) \leq 3e^{-c n t^2} 
+e^{-n (1/2-\alpha)^2} 
,\end{align}
where  $q_{1-\alpha+t}$ is the $(1-\alpha+t)$-quantile of the distribution of $\xi$. 
\end{lemma}

\begin{lemma}
    \label{lem:ub_z2}
Let $z_1,\ldots,z_n$ be i.i.d.\ standard Gaussians, and, independently, 
let $\eps_1,\ldots,\eps_n$ be i.i.d.\ realizations of a centered, sub-Gaussian random variable
$\eps$.
 Let $\cS$ be the set containing the indices of the smallest $n-k$ values in
    $ \{\eps_i z_{i}\}_{i\in [n]}$.
Then, for any $\alpha \coloneqq k/n<1/2$ and $t \in (0, \alpha)$,
    \begin{align*}
        \P\left(\frac{1}{n}\sum_{i\in \cS}  z_{i}^2 \geq 1-\alpha+ 3t   \right) \leq 3 e^{-cnt^2}.
    \end{align*}
   \end{lemma}

\begin{lemma}
    \label{lem:lb_z2}
 Under the setting of Lemma \ref{lem:ub_z2},
    for any $t \in (0, \alpha)$,
    \begin{align*}
        \P\left(\frac{1}{n-k}\sum_{i\in \cS}  z_{i}^2 \leq \frac{1}{4} \right) \leq e^{-cn} + 2e^{-2n(1/2-\alpha)^2} . 
    \end{align*}
\end{lemma}

\label{appendix:LB}

\subsection{Proof of Theorem \ref{thrm:small_p_large_k}} \label{secA4}

\begin{proof}[Proof of Theorem \ref{thrm:small_p_large_k}]

Similarly to (\ref{g}),
\begin{align}
\widehat \bbeta  - \widehat \bbeta_\cS  = \Sigma^{-1/2}\Big(
\big(Z^\top \spa Z\big)^{-1}Z^\top \beps - 
\big(Z_\cS^\top Z_\cS\big)^{-1}Z_{\cS}^\top \beps_\cS\Big).   \nonumber 
\end{align}
Substituting this expression into (\ref{eq:delta_v}), the definition of $\Delta_k (\bv)$, yields
\begin{align}
\Delta_k (\bv) &= \max_{\cS \subset [n], |\cS| = n-k} \Big \langle 
\big(Z^\top  \spa Z\big)^{-1}Z^\top \beps - 
\big(Z_\cS^\top Z_\cS\big)^{-1}Z_{\cS}^\top \beps_\cS, \Sigma^{-1/2}\bv 
\Big\rangle \nonumber.
\end{align}
Since the Gaussian distribution is rotationally invariant,
\begin{align} 
\Delta_k (\bv) & \stackrel{\,d\,}{=} \big\|\Sigma^{-1/2}\bv\big\| \max_{\cS \subset [n], |\cS| = n-k} \Big \langle 
\big(Z^\top  \spa Z\big)^{-1}Z^\top \beps - 
\big(Z_\cS^\top Z_\cS\big)^{-1}Z_{\cS}^\top \beps_\cS,  \be_1 \Big\rangle . \nonumber
\end{align}
Thus, it suffices to prove (\ref{eq:small_p_large_k}) for 
\begin{align} \label{pdg1x7} 
\widetilde\Delta_k \coloneqq \be_1^\top \big(Z^\top  \spa Z\big)^{-1}Z^\top \beps - \min_{\cS \subset [n], |\cS| = n-k}\be_1^\top \big(Z_\cS^\top Z_\cS\big)^{-1}Z_{\cS}^\top \beps_\cS
\end{align}
under the normalization $\big\|\Sigma^{-1/2}\bv\big\| = 1$.

The first term on the right-hand side of (\ref{pdg1x7}) is 
negligible. Indeed, applying Lemma \ref{lem:bern} as in (\ref{bf123c1sdjkf}) and recalling $\eta\coloneqq \sqrt{2\E |\eps|^2 + \|\eps\|_{\psi_2}^2}$,
\begin{align}
    \P\big(\|\beps\|\geq \eta \sqrt{ n} \,\big)  \leq e^{-cn}\label{Bbern} .
\end{align}
Combining this with Lemma
\ref{lemma:ub_ZTZ_e1}  yields that, for any $t\in \big(0, 1-\sqrt{p/(n-1)}\,\big),$
\begin{align}
        \P\left(\be_1^\top (Z^\top  \spa Z)^{-1}Z^\top \beps\leq - \frac{\eta \hspace{.1em} t }
        {\big(1-t-\sqrt{p/(n-1)}\,\big)^2}\right) \leq e^{-cn} + 2e^{-nt^2/2} . \label{q2u3e}
    \end{align}

Next, we construct a subset $\cS$ such that, with high probability, $-\be_1^\top \big(Z_\cS^\top Z_\cS\big)^{-1}Z_{\cS}^\top \beps_\cS \geq C$, for some constant $C > 0$.
This yields a lower bound on $\widetilde \Delta_k$ since
\begin{align}
\label{eq:lb_deltak}
    \widetilde\Delta_k  \geq  \be_1^\top \big(Z^\top  \spa Z\big)^{-1}Z^\top \beps - \be_1^\top \big(Z_\cS^\top Z_\cS\big)^{-1}Z_{\cS}^\top \beps_\cS.
\end{align}
Specifically, let $\cS$ contain the indices of the smallest $n-k$ values in 
$ \{\eps_i z_{i1}\}_{i\in [n]}$.
By (\ref{qwertyuiop1}),
\begin{align}
\label{eq:id_e1_inv(ZsZs)_Zs_eps}
    \be_1^\top \big(Z_\cS^\top Z_\cS\big)^{-1}Z_{\cS}^\top \beps_\cS\overset{d}{=}  \be_1^\top \big(Z_\cS^\top Z_\cS\big)^{-1} \be_1
    \big( Z_{\cS,1}^\top P \beps_\cS\big),
\end{align}
where $Z_{\cS,1} \coloneqq Z_\cS \be_1$ and $P$ is the orthogonal projection 
onto a subspace drawn uniformly from the
Grassmannian $G_{n-k,n-k-p+1}$, independent of
$(\cS, Z_{\cS,1}, \beps_\cS)$.

We now prove that $ Z_{\cS,1}^\top P \beps_\cS$ is negative with high probability.  
Applying Lemma~\ref{lemma:JL_product}, for any $\delta \in (0,1]$,
\begin{align}
  \P  \bigg(  Z_{\cS,1}^\top P \varepsilon_\cS \geq (1 - \gamma) \Big( Z_{\cS,1}^\top \beps_\cS
+ \delta \|Z_{\cS,1}\|\|\beps_\cS\|\Big) \,\Big|\, Z_{\cS,1}, \beps \bigg) \leq 2 e^{-c(n-k-p+1)\delta^2} . \label{jjjjj123}
\end{align}
Recall $Q_{\alpha, t} := (1-t)\E\big[ \varepsilon z \,\ind(\varepsilon z>q_{1-\alpha+t})\big] - t \| \varepsilon z \|_{\psi_1}$. By Lemma~\ref{lemma:lb_sum_w1}, 
\begin{align}
      \P \bigg( \frac{1}{n} Z_{\cS,1}^\top \beps_\cS  \geq -Q_{\alpha,t}\bigg) \leq  3e^{-cnt^2} + e^{-n (1/2-\alpha)^2 }. \label{jjjjj1234}
\end{align}
Combining (\ref{jjjjj123}) and (\ref{jjjjj1234}) and using $\|\beps_\cS\| \leq \|\beps\|$, we obtain
\begin{align}
  \P  \bigg(  Z_{\cS,1}^\top P \varepsilon_\cS \geq -(1 - \gamma) \Big(n Q_{\alpha,t}
- \delta \|Z_{\cS,1}\|\|\beps\|\Big)  \bigg) \leq 3e^{-cnt^2} + e^{- n(1/2-\alpha)^2 } + 2 e^{-c(n-k-p+1)\delta^2} . \nonumber
\end{align}
By Lemmas \ref{lem:ub_z2} and \ref{lem:lb_z2} and  (\ref{Bbern}), with probability at least $1 - 5e^{-cnt^2} -2 e^{- n(1/2-\alpha)^2 }$,
\begin{align}
    \frac{n Q_{\alpha,t}}{\|Z_{\cS,1}\|^2} - \frac{\delta \|\beps\|}{\|Z_{\cS,1}\|} & \geq  \frac{Q_{\alpha,t}}{1-\alpha + 3t} - 2 \hspace{.05em} \eta \hspace{.05em} \delta\sqrt{\frac{n}{n-k}} \geq \frac{Q_{\alpha,t}}{1-\alpha+3t} - 3 \hspace{.05em} \eta \hspace{.05em} \delta , \nonumber
\end{align}
where the right-hand side of this inequality is strictly positive---see Remark \ref{rem:c_eps}. Using
 (\ref{qwertyuiop}),
\begin{gather} \nonumber
\P\bigg(  
\|Z_{\cS,1}\|^2 \be_1^\top \big(Z_\cS^\top Z_\cS\big)^{-1} \be_1
\leq
\frac{1}{1-\gamma+t} \bigg) \leq  2 e^{-c(n-k-p+1)t^2} .    \nonumber
\end{gather}
From the above bounds and (\ref{eq:id_e1_inv(ZsZs)_Zs_eps}), we conclude that with probability at least
\[
1 - 10e^{-cnt^2} -  3e^{- n(1/2-\alpha)^2 } - 2 e^{-c(n-k-p+1)\delta^2} ,
\]
the following inequality holds:
\begin{align}
    \label{fuck_regression}
    \be_1^\top \big(Z_\cS^\top Z_\cS\big)^{-1}Z_{\cS}^\top \beps_\cS \leq - \frac{1-\gamma}{1-\gamma+t}\bigg(\frac{Q_{\alpha,t}}{1-\alpha+3t} - 3 \hspace{.05em} \eta \hspace{.05em} \delta \bigg) . 
\end{align}
The result follows from (\ref{q2u3e}), (\ref{eq:lb_deltak}), and (\ref{fuck_regression}).

\end{proof}

\subsection{Proofs of Lemmas \ref{lemma:ub_ZTZ_e1}--\ref{lem:lb_z2}}
\label{APBLEMM}
\begin{proof}[Proof of Lemma  \ref{lemma:ub_ZTZ_e1}]
By the rotational invariance of the Gaussian distribution,
\begin{align}
\be_1^\top (Z^\top Z)^{-1}Z^\top \bv 
\stackrel{d}{=}
\be_1^\top (Z^\top Z)^{-1}Z^\top \be_1
=
\be_1^\top (Z^\top Z)^{-1}\bz_1, \label{qwdq61}
\end{align}
where $\bz_1$ is the first row of $Z$.
Let $Z_{-1} \in \mathbb{R}^{(n-1) \times p}$ be the matrix obtained from $Z$ by removing $\bz_1$, and set $S_{-1} \coloneqq Z_{-1}^\top Z_{-1}$. 
Since $S_{-1}$ is almost surely positive definite, the Sherman--Morrison formula gives
    \begin{align}
     & (Z^\top  \spa Z)^{-1} = S_{-1}^{-1} - \frac{S_{-1}^{-1}\bz_1\bz_1^\top S_{-1}^{-1}}{1+\bz_1^\top S_{-1}^{-1}\bz_1} , &&
        \be_1^\top (Z^\top  \spa Z)^{-1}\bz_1 
        =
        \frac{\be_1^\top 
        S_{-1}^{-1}\bz_1}{1+\bz_1^\top 
        S_{-1}^{-1}\bz_1}
          . \label{c15asdfq}
       \end{align}

\noindent Observe that $\bz_1 \mapsto \be_1^\top S_{-1}^{-1}\bz_1$ is $\sigma_{p}^{-2}(Z_{-1})$-Lipschitz. Hence,  conditioning on $Z_{-1}$ and applying Lemma~\ref{lem:lipschitz}, for
all $t\ge 0$,
\begin{align*}
    \P\bigg( \be_1^\top S_{-1}^{-1} \bz_1 \geq \E \Big[  \be_1^\top S_{-1}^{-1} \bz_1 \, \big| \, Z_{-1}\Big] + \sigma_p^{-2}(Z_{-1})  \hspace{.1em} t \, \Big| \, Z_{-1} \bigg) \leq e^{-t^2/2} . 
\end{align*}
Since $\E \big[  \be_1^\top S_{-1}^{-1} \bz_1 \, \big| \, Z_{-1}\big] =0$, this implies
\begin{align}
    \P\Big( \be_1^\top S_{-1}^{-1} \bz_1 \geq \sigma_p^{-2}(Z_{-1})  \hspace{.1em} t \Big) =  \P\Big( \be_1^\top S_{-1}^{-1} \bz_1 \leq -\sigma_p^{-2}(Z_{-1})  \hspace{.1em} t \Big) \leq e^{-t^2/2} . \label{c17aqwg}
\end{align}
Combining  (\ref{c15asdfq}) and (\ref{c17aqwg}), we obtain
\begin{align}
  &~  \P\Big( \be_1^\top (Z^\top  \spa Z)^{-1} \bz_1 \geq -\sigma_p^{-2}(Z_{-1})  \hspace{.1em} t \Big) \geq  \P\Big(  \be_1^\top S_{-1}^{-1} \bz_1  \geq -\sigma_p^{-2}(Z_{-1})  \hspace{.1em} t  \Big)  \geq 1 - e^{-t^2/2}. \label{qwdq62}
\end{align}
The result follows from (\ref{qwdq61}), (\ref{qwdq62}), and Lemma  \ref{Lem:maxeig2}.

\comment{
By Lemma \ref{Lem:maxeig2}, for $t \in (0,\sqrt{p/(n-1)})$,
\begin{align}
    \P\bigg( \be_1^\top S_{-1}^{-1} \bz_1 \geq \frac{t}{\sqrt{n-1}(1-t-\sqrt{p/(n-1)})^2} \bigg) \leq 2 e^{-(n-1)t^2/2}. 
\end{align}
    For any $\delta > 0$, since $1+\bz_1^\top  S_{-1}^{-1}\bz_1\geq 1$ almost surely,
    \begin{align}
    \label{eq:ub_P}
            \P\Big( \be_1^\top (Z^\top  \spa Z)^{-1}\bz_1\geq -\delta\Big)
            \geq \P\big(\be_1^\top 
        S_{-1}^{-1}\bz_1\geq -\delta
        \big) = 
        \P\big(\be_1^\top 
        S_{-1}^{-1}\bz_1\leq \delta
        \big).
    \end{align}

Thus,
\[
\P\left( \be_1^\top (Z^\top  \spa Z)^{-1}\bz_1\geq -\frac{t}{\sqrt{n-1}(1-t-\sqrt{p/(n-1)}\,)^2}\right) \geq 1 - 2 e^{-(n-1)t^2/2}.
\]}
\end{proof}

\begin{proof}[Proof of Lemma~\ref{lemma:id_e1_inv(ZsZs)_Zs_eps}]
Write $
Z_\cS = \begin{bmatrix} Z_{\cS,1} & Z_{\cS,-1}\end{bmatrix} $, where $Z_{\cS,1} \coloneqq Z_{\cS} \be_1  \in \mathbb{R}^{n-k}$ is the first column of $Z_\cS$ and $Z_{\cS,-1} \in \mathbb{R}^{(n-k) \times (p-1)}$ contains the remaining columns. Let $P_\cS$ denote the projection onto the orthogonal complement of the column space of $Z_{\cS,-1}$:
\[
P_\cS := I_{n-k} - Z_{\cS,-1} \big(Z_{\cS,-1}^\top Z_{\cS,-1}\big)^{-1}Z_{\cS,-1}^\top .
\]
Consider the block
decomposition
\[
Z_\cS^\top Z_\cS
=
\begin{pmatrix}
Z_{\cS,1}^\top Z_{\cS,1} & Z_{\cS,1}^\top Z_{\cS,-1}\\
Z_{\cS,-1}^\top Z_{\cS,1} & Z_{\cS,-1}^\top Z_{\cS,-1}
\end{pmatrix}.
\]
By the block matrix inverse formula, the first row of $(Z_\cS^\top Z_\cS)^{-1}$ equals
\begin{align*}
\be_1^\top (Z_\cS^\top Z_\cS)^{-1}
&=
\be_1^\top\big( Z_{\cS}^\top Z_{\cS} \big)^{-1} \be_1 \begin{bmatrix} 1 & -Z_{\cS,1}^\top Z_{\cS,-1} \big(Z_{\cS,-1}^\top Z_{\cS,-1}\big)^{-1} \end{bmatrix}  
\\
& =
\big( Z_{\cS,1}^\top P_\cS Z_{\cS,1} \big)^{-1}  \begin{bmatrix} 1 & -Z_{\cS,1}^\top Z_{\cS,-1} \big(Z_{\cS,-1}^\top Z_{\cS,-1}\big)^{-1} \end{bmatrix} . 
\end{align*}
Multiplying by $Z_\cS^\top \beps_\cS =  \begin{bmatrix} Z_{\cS,1}^\top \beps_\cS &  Z_{\cS,-1}^\top \beps_\cS \end{bmatrix}$, we obtain 
\begin{align}
\be_1^\top (Z_\cS^\top Z_\cS)^{-1}Z_\cS^\top \beps_\cS
=
\be_1^\top\big( Z_{\cS}^\top Z_{\cS} \big)^{-1} \be_1 Z_{\cS,1}^\top P_\cS \beps_\cS . 
\label{boaqd}\end{align}
Since $\cS$ depends only on
$\{\eps_i z_{i1}\}_{i\in[n]}$,  $Z_{\cS,-1}$ and $P_\cS$ are independent
of $(\cS,Z_{\cS,1},\beps_\cS)$. Writing the singular value decomposition of $Z_{\cS,-1}$ as $Z_{\cS,-1}=U\Sigma V^\top$, where $U \in \mathbb{R}^{(n-k) \times (p-1)}$, we have $P_\cS = I_{n-k} - UU^\top$, which is the orthogonal projector onto a
subspace  drawn uniformly from
$G_{n-k,n-k-p+1}$ by Gaussian rotational invariance. 

Using the independence of $Z_{\cS,1}$ and $P_\cS$ and the identity \[ \frac{1}{\be_1^\top\big( Z_{\cS}^\top Z_{\cS} \big)^{-1}  \be_1} =  Z_{\cS,1}^\top P_\cS Z_{\cS,1} = \big\|P_\cS Z_{\cS,1} \big\|^2, \]     Lemma 5.3.2 of \cite{Vershynin12}
implies that for $t > 0$,
\begin{align}\nonumber 
\P\bigg( \bigg| \Big( \|Z_{\cS} \be_1\|^2\be_1^\top \big(Z_\cS^\top Z_\cS\big)^{-1} \be_1\Big)^{-1/2} - \sqrt{1-\gamma}\bigg| \geq t \sqrt{1-\gamma} \hspace{.1em}\bigg) \leq 2 e^{-c(n-k-p+1)t^2} . 
\end{align}
Equivalently,
\begin{align}\nonumber 
\P\bigg( \bigg|
\Big(\|Z_{\cS}\be_1\|^2\,\be_1^\top (Z_\cS^\top Z_\cS)^{-1}\be_1\Big)^{-1}
-
(1 - \gamma)
\bigg|
\geq
\big(2t+t^2\big)(1-\gamma) \bigg) \leq 2 e^{-c(n-k-p+1)t^2} . 
\end{align}
This implies 
  (\ref{qwertyuiop}) since for $t \in (0,1)$, $\big(2t+t^2\big)(1- \gamma) \leq 3 t $.

\end{proof}

\begin{proof}[Proof of Lemma \ref{lemma:JL_product}]
    By Lemma 5.3.2 in  \cite{Vershynin12}, the following bounds hold individually  with probability at least $1 - e^{-cp\delta^2}$:
        \begin{align*}
   &  \|P(\bv+\bw)\|\leq \Big(1+\frac{\delta}{2}\Big) \sqrt{\frac{p}{n}}\|\bv+\bw\|,  & &\|P(\bv-\bw)\|\geq  \Big(1-\frac{\delta}{2}\Big)\sqrt{\frac{p}{n}}\|\bv-\bw\| .
    \end{align*}
Thus,  with probability at least $1 - 2e^{-cp\delta^2}$,
        \begin{align}
        \bv^\top P \bw & = \frac{1}{4} \left( \|P(\bv+\bw)\|^2 - \|P(\bv-\bw)\|^2 \right) \nonumber\\
        &  \leq  \frac{p}{4n}\left(\Big(1+\frac{\delta}{2}\Big)^2\|\bv+\bw\|^2 - \Big(1-\frac{\delta}{2}\Big)^2\|\bv-\bw\|^2 \right) \nonumber \\
        &= \frac{p}{n}\left(\Big(1+\frac{\delta^2}{4}\Big)\bv^\top \bw +\frac{\delta}{2}\big(\|\bv\|^2+\|\bw\|^2\big)\right) \nonumber 
        \leq  \frac{p}{n} \left( \bv^\top \bw  +\delta\right) \nonumber.
    \end{align}
    Here, the first equality  uses  $P^\top \spa P = P$.
\end{proof}

The proof of Lemma~\ref{lemma:lb_sum_w1} uses the following Chernoff
bound for sums of Bernoulli random variables:
\begin{lemma}[\citet{chernoff_1952}]
\label{lemma:chernoff}
     Let $x_1, \ldots, x_n$ be i.i.d.\ Bernoulli(q) random variables.
     Then, 
    for any $t \in (0,1]$,
     \begin{align*}
         \P\left( \frac{1}{n}\sum_{i=1}^n x_i \leq (1-t)
         q\right) \leq e^{- n q  t^2 /2}.
     \end{align*}
\end{lemma}

\begin{proof}[Proof of Lemma~\ref{lemma:lb_sum_w1}]

First, we will show that with high probability,
 \begin{align}\label{eq:lb_sum_wi}
 \sum_{i = 1}^{n-k} \xi_{(i)}  =  \sum_{i = 1}^{n} \xi_{i} \ind(\xi_i \leq \xi_{(n-k)})  \leq  \sum_{i=1}^n \xi_i \ind(\xi_i \leq q_{1-\alpha+t}) . \end{align}
 We then bound the right-hand side using Lemma \ref{lem:bern},
as it is a sum of i.i.d.\ random variables.

 Since the event $ \big\{\xi_{(n-k+1)} \in [0, q_{1-\alpha+t}] \big\}$ implies (\ref{eq:lb_sum_wi}),
\begin{align}
    \P \bigg( \sum_{i = 1}^{n-k} \xi_{(i)}  \leq  \sum_{i=1}^n \xi_i \ind(\xi_i \leq q_{1-\alpha+t}) \bigg) \geq 1 - \P\big(\xi_{(n-k+1)}< 0\big)-\P\big(\xi_{(n-k+1)}> q_{1-\alpha+t}\big) .  \label{eq:lb_Pa123asd}
\end{align}
Let $F(x) \coloneqq \P(\xi \leq x)$ denote the distribution function of $\xi$.
Using the identity $\ind(\xi> q) = 1 - \ind(\xi \leq  q) $,  the definition $\alpha \coloneqq k/n$, and the fact that $F(q_{1-\alpha+t}) \geq 1 -\alpha +t$, we obtain
\begin{align} 
      \P \big (\xi_{(n-k+1)}> q_{1-\alpha+t} \big)  & =   
  \P\left( 
   \frac{1}{n}\sum_{i=1}^n\ind(\xi_i\leq q_{1-\alpha+t}) \leq 1 - \alpha
   \right)   \nonumber \\
   & \leq  \P\left( 
   \frac{1}{n}\sum_{i=1}^n\ind(\xi_i\leq q_{1-\alpha+t}) \leq F(q_{1-\alpha+t}) - t   \right)   \nonumber .
\end{align}
By the Dvoretzky--Kiefer--Wolfowitz inequality \citep{massart1990tight},
\begin{align*}
\P\!\left(
\sup_{x\in\mathbb R}
\left|
\frac{1}{n}\sum_{i=1}^n \ind(\xi_i\le x)-F(x)
\right|
\ge t
\right)
\le
2e^{-2nt^2}.
\end{align*}
Together with~\eqref{kqjxh72}, this yields
\begin{align}
\P\big(\xi_{(n-k+1)}>q_{1-\alpha+t}\big)
\le
2e^{-2nt^2}.
 \label{kqjxh72}
\end{align}

To bound $\P(\xi_{(n-k+1)}<0)$, we use the symmetry of 
$\xi$, which implies $\P(\xi \geq 0) \geq 1/2$, and apply a Chernoff bound
(Lemma~\ref{lemma:chernoff}): 
\begin{align}
 \P\big(\xi_{(n-k+1)}<0\big)   &= \P\Bigg(\frac{1}{n}\sum_{i=1}^n \ind(\xi_i \geq 0) < \alpha \Bigg)
    \leq e^{-n(1/2-\alpha)^2}. \label{qwerty1asdf2}
\end{align}
Combining \eqref{eq:lb_Pa123asd}, \eqref{kqjxh72}, and  \eqref{qwerty1asdf2} yields
 \begin{align} \label{g7w6tefwgef78}
  \P \bigg( \sum_{i = 1}^{n-k} \xi_{(i)}  \geq  \sum_{i=1}^n \xi_i \ind(\xi_i \leq q_{1-\alpha+t}) \bigg) \leq   2e^{-2nt^2}  + e^{-n(1/2-\alpha)^2} ,
 \end{align}
establishing that (\ref{eq:lb_sum_wi}) holds with high probability.

Next, we bound $\sum_{i=1}^n \xi_i \ind(\xi_i \leq q_{1-\alpha+t})
$. Note that the random variable
$\xi \ind(\xi \le q)-\E[\xi \ind(\xi \le q)]$ is centered and sub-exponential.
By the monotonicity of the Orlicz norm and the symmetry of $\xi$,
\begin{align*}
\big\|\xi \ind(\xi \le q)-\E[\xi \ind(\xi \le q)]\big\|_{\psi_1}
&\le
\|\xi \ind(\xi \le q)\|_{\psi_1}
+ \big|\E[\xi \ind(\xi \le q)]\big| \le
\|\xi\|_{\psi_1}-\E[\xi \ind(\xi \le q)] . 
\end{align*}
Thus, by Lemma \ref{lem:bern}, 
     \begin{align}
        \P \left(\frac{1}{n} \sum_{i=1}^n \xi_i \ind(\xi_i \leq q_{1-\alpha+t})
        \geq 
        (1-t)\E\big[ \xi \ind(\xi\leq q_{1-\alpha+t})\big]+ t \|\xi\|_{\psi_1} \right) \leq e^{-c n t^2}. \nonumber
    \end{align}
Equivalently, 
        \begin{align}
        \label{eq:lb_sub_exp_sum}
        \P \left(\frac{1}{n} \sum_{i=1}^n \xi_i \ind(\xi_i \leq q_{1-\alpha+t})
        \geq 
        -(1-t)\E\big[ \xi \ind(\xi\geq q_{1-\alpha+t})\big]+ t \|\xi\|_{\psi_1} \right) \leq e^{-c n t^2}. 
    \end{align}
The result follows (\ref{g7w6tefwgef78}) and (\ref{eq:lb_sub_exp_sum}). 
\end{proof}


\begin{lemma} \label{lemma:ub_Ez|t}
Let $z \sim \mathcal{N}(0,1)$ and let $\eps$ be an independent random variable. For any $t \geq 0$,
\[
\E\big[z^2 \ind(\eps z \leq t)\big] \le \P(\eps z \le t).
\]
\end{lemma}

\begin{proof}[Proof of Lemma \ref{lem:ub_z2}] Let $q_{1-\alpha+t}$ denote the $(1-\alpha+t)$-quantile of $\eps z$, where $\eps$ and $z \sim \mathcal{N}(0,1)$ are independent. 
Applying ~\eqref{kqjxh72} to $\{\eps_i z_i \}_{i \in [n]}$, we have that for any $t \in (0,\alpha)$,
\[
\P\bigg(\sum_{i=1}^n \ind(\eps_i z_i> q_{1-\alpha + t}) \geq k     \bigg) \leq 2 e^{-2nt^2}. 
\]
Similarly to (\ref{eq:lb_sum_wi})--(\ref{eq:lb_Pa123asd}), this bound implies 
\begin{align} \label{123cjsdfvs2gteg}
   \P \bigg( \sum_{i \in \cS} z_i^2 \leq \sum_{i=1}^n z_i^2  \ind(\eps_i z_i \leq q_{1-\alpha+t}) \bigg) \geq 1 - 2 e^{-2nt^2} . 
\end{align}
By Lemma~\ref{lem:bern},
\begin{align}
    \label{eq:ub_sum_zi^2}
        \P\left(
        \frac{1}{n} \sum_{i=1}^n z^2_i \ind(\eps_i z_i \leq q_{1-\alpha+t})
        -
        \E\big[z^2 \ind (\eps z \leq q_{1-\alpha+t})\big]
        \geq
         2 t \right) \leq e^{-cnt^2},
    \end{align}
where we used the bound
\begin{align} \label{123g7wf}
\big\|z^2 \ind(\eps z \leq q_{1-\alpha+t})
        -
        \E\big[z^2 \ind (\eps z \leq q_{1-\alpha+t})\big]\big\|_{\psi_1} \leq \|z\|_{\psi_2}^2 + \E| z|^2 \leq 2 . 
\end{align}
    %
    The result follows from (\ref{123cjsdfvs2gteg}), (\ref{eq:ub_sum_zi^2}), and Lemma \ref{lemma:ub_Ez|t}, which implies
\[ \E\big[z^2 \ind(\eps z \leq q_{1-\alpha+t})\big] \leq \P(\eps z \leq q_{1-\alpha+t}) = 1 - \alpha + t . \]
\end{proof}


\begin{proof}[Proof of Lemma \ref{lem:lb_z2}] 

As the distribution of $\varepsilon z$ is symmetric, applying  (\ref{qwerty1asdf2}) to $\{\eps_i z_i\}_{i \in [n]}$ yields
\[ \P\left( 
   \sum_{i=1}^n\ind(\eps_i z_i \geq 0) <  k
   \right)\leq 2e^{-n(1/2-\alpha)^2} . 
\]
Similarly to (\ref{eq:lb_sum_wi})--(\ref{eq:lb_Pa123asd}), this bound implies 
\begin{align} \label{e21asde5hu5}
    \P \bigg( \sum_{i \in \cS} z_i^2 \geq \sum_{i=1}^n z^2_i \ind(\eps_i z_i \leq 0 ) \bigg) \geq 1 - 2e^{-n(1/2-\alpha)^2} .
\end{align}

\noindent By Lemma~\ref{lem:bern} and (\ref{123g7wf}), for any $t \in (0, \alpha)$, 
\begin{align*}
    \P \bigg( \frac{1}{n} \sum_{i=1}^n z^2_i \ind(\eps_i z_i\leq 0) - \E\big[z^2 \ind (\eps z \leq 0)\big] \leq - 2 t \bigg) \leq e^{- c n t^2} . 
\end{align*}
Since $\eps$ and $z$ are independent, 
\[\E\big[z^2 \ind (\eps z  \leq 0)\big] \geq \frac{1}{2}\E\big[z^2 | \eps z  \leq 0\big] 
= \frac{1}{2} . \] 
Thus,
\begin{align} \label{123dfbhuf4dt24thgwrn}
    \P \bigg( \frac{1}{n-k} \sum_{i=1}^n z^2_i \ind(\eps_i z_i\leq 0) \leq \frac{1}{2} - 2 t \bigg) \leq e^{- c n t^2} .
\end{align}
The result follows from (\ref{e21asde5hu5}) and (\ref{123dfbhuf4dt24thgwrn}).
\comment{
To complete the proof, we now lower bound the sum $T_0 = \sum_{i=1}^n z^2_i \ind(\xi_i\leq 0)$.
By Lemma~\ref{lem:bern} there exists an absolute constant $c_1>0$ such that for any $t\in(0,1)$
    \begin{align}
        \P\left(
        T_0/n - \E[z^2 \ind (\eps z \leq 0) ] \leq -Kt  \right) \leq e^{-c_1nt^2},\nonumber
    \end{align}
where     $K = \|z^2 \ind (\eps z \leq 0) - \E[z^2 \ind (\eps z \leq0)]\|_{\psi_1}\leq 1$.
Since $\eps z $ has the same distribution as $|\eps|\,z$, 
\[
\E[z^2 \ind (\eps z <0)] = \frac{1}{2} \E[z^2 | |\eps| z <0] = \frac{1}{2}\E[z^2|z<0] = \frac{1}{2}.
\]
Hence,
\begin{align}
        \P\left(
        T_0/n\leq 1/2-3t/2  \right) \leq e^{-c_1nt^2}. \nonumber
\end{align}
    Multiplying both sides by $\frac{1}{1-\alpha}$ gives
    \begin{align}
         \P\left(
        T_0/(n-k) \leq 
         (1-\alpha)^{-1} (1/2-3t/2)\right) \leq e^{-c_1nt^2}. \nonumber
    \end{align}
    Since $(1-\alpha)^{-1} > 1$, 
    \[
    \P\left(
        T_0/(n-k) \leq 
        (1/2-3t/2)\right) \leq e^{-c_1 nt^2}.
    \]
    Choosing $t = 1/6$ and setting $c = c_1 / 36$ gives
    \begin{align}
         \P\left(
        T_0/(n-k) \leq 
        1/4 \right) \leq e^{-cn}.
    \end{align}
    The result follows by combining this with \eqref{eq:lb_sum_z^2_2}, and using a union bound.}
\end{proof}

\comment{\begin{lemma} 
\label{lemma:ub_Ez|t}
    Let $z\sim \mathcal{N}(0,1)$. Then, for any $t\geq0$, $ \E[z^2| z\leq t] \leq 1 $.
\end{lemma}
\begin{proof}
    Let $\phi(s)$ be the probability density function of the standard normal distribution.
    By definition,
    \[
    \E[z^2 \,|\,z<t] = \frac{\int_{-\infty}^t s^2 \phi(s) ds}{\int_{-\infty}^t \phi(s) ds}  = \frac{\int_{-\infty}^t s^2e^{-s^2/2}ds}{\int_{-\infty}^t e^{-s^2/2}ds}.
    \]
    Integrating by parts the numerator gives
    \[
    {\int_{-\infty}^ts^2 e^{-s^2/2} ds} = -te^{-t^2/2} + {\int_{-\infty}^t e^{-s^2/2} ds} .
    \]
    Thus,
    \[
    \E[z^2 \,|\,z\leq t] = 1- \frac{te^{-t^2/2}}{{\int_{\infty}^t e^{-s^2/2} ds}} \leq 1.
    \]
\end{proof}}

\begin{proof}[Proof of Lemma \ref{lemma:ub_Ez|t}]
By the independence of $\eps$ and $z$,  
\[
\E\big[z^2 \ind(\eps z \leq t)\big] = \int_{\mathbb{R}^2} z^2 \ind(\eps z \le t) \, \phi(z) \,  dF_\eps(x) \, dz,
\]
where $F_\eps$ is the distribution function of $\eps$ and $\phi(z)$ denotes the standard normal density. For each $x \in \mathbb{R}$, define  
\[
A_x := \{ z \in \mathbb{R} :  x z \le t\} =
\begin{cases}
(-\infty, t/x] & x > 0 \\
[t/x, \infty) & x < 0 \\
\mathbb{R} & x = 0 
\end{cases} .
\]  Then, 
\begin{align} \nonumber
    \E\big[z^2 \ind(\eps z \leq t)\big] = \int_{\mathbb{R}} \int_{A_x} z^2 \, \phi(z) \, dz \,  dF_\eps(x) . 
\end{align}
Recall that for the standard normal density,
\[
\int_{-\infty}^x z^2 \phi(z) \, dz \leq \int_{-\infty}^x \phi(z) \, dz \quad \text{for } x > 0, \quad
\int_x^{\infty} z^2 \phi(z) \, dz \leq \int_x^{\infty} \phi(z) \, dz \quad \text{for } x \leq 0.
\]
Thus,
\[
\int_{A_x} z^2 \phi(z) dz \le \int_{A_x} \phi(z) dz.
\]
It follows that 
\[
\E\big[z^2 \ind(\eps z \leq t)\big] \le \int_\mathbb{R} \int_{A_x} \phi(z) \, dz \, dF_\eps(x) = \P(\eps z \le t),
\]
completing the proof.
\end{proof}


\end{document}